%% file: nonlinear_flp.tex
\documentclass[10pt]{article}
\usepackage{amsmath,amssymb,amsthm,graphics,graphicx,color,rotating,enumerate}
\usepackage[all]{xy}

\newtheorem{theorem}{Theorem}
\newtheorem*{theorem*}{Theorem}

\newtheorem{lemma}{Lemma}
\newtheorem{proposition}{Proposition}
\newtheorem{corollary}{Corollary}
\theoremstyle{definition}

\newtheorem{remark}{Remark}

\numberwithin{equation}{section}
\usepackage{geometry}
\newcommand{\R}{\ensuremath{\mathbf{R}}}

\newcommand{\defn}{\ensuremath{\overset{\mathrm{def}}{=}}}

\newcommand{\dd}{\ensuremath{\mathrm{d}}}
\newcommand{\ii}{\ensuremath{\mathrm{i}}}

\newcommand{\captionfonts}{\small}
\makeatletter  % Allow the use of @ in command names
\long\def\@makecaption#1#2{%
  \vskip\abovecaptionskip
  \sbox\@tempboxa{{\captionfonts #1: #2}}%
  \ifdim \wd\@tempboxa >\hsize
    {\captionfonts #1: #2\par}
  \else
    \hbox to\hsize{\hfil\box\@tempboxa\hfil}%
  \fi
  \vskip\belowcaptionskip}
\makeatother   % Cancel the effect of \makeatletter

\begin{document}
\title{\bf{Collections of Fluid Loaded Plates: \\
A Nonlocal Approach}}
\author{A.C.L Ashton\footnote{a.c.l.ashton@damtp.cam.ac.uk} \\
\\
Department of Applied Mathematics \\
and Theoretical Physics,\\
University of Cambridge, \\
CB3 0WA, UK.}
\maketitle

\begin{abstract}
We consider the motion of a collection of fluid loaded elastic plates, situated horizontally in an infinitely long channel. We use a new, unified approach to boundary value problems, introduced by A.S. Fokas in the late 1990's, and show the problem is equivalent to a system of 1-parameter integral equations. We give a detailed study of the linear problem, providing explicit solutions and well-posedness results in terms of standard Sobolev spaces. We show that the associated Cauchy problem is completely determined by a matrix, which depends solely on the mean seperation of the plates and the horizontal velocity of each of the driving fluids. This matrix corresponds to the infinitesimal generator of the $C_0$-semigroup for the evolution equations in Fourier space. By analysing the properties of this matrix, we classify necessary and sufficient conditions for which the problem is asymptotically stable.
\end{abstract}

\newpage
\section{Introduction}
In this paper we are concerned with the motion of a collection of elastic plates residing in a horizontal channel. Between each plate is a fluid with mean flow in the horizonal direction which drive the motion of the elastic plates. 

The physical significance of the problem is discussed in detail in \cite{jia2007cmb}, where the authors give both a theoretical and experimental study of the flapping modes of two elastic bodies of finite length. In this paper, following \cite{crighton1991flm}, we do not impose any external length scale. As such, the results are particularly applicable to cases in which the elastic bodies are long: an example being the motion of underwater cables.

At the linear level, many approaches to similar problems see authors choosing to work with the case in which the channel is infinite in \emph{both} the vertical and horizontal directions. This approach often serves as a good approximation to some physically interesting scenarios, and also simplifies the analysis if working with a combination of Laplace and Fourier transforms. The use of the Laplace transform is generally inappropriate for the following reason: (a) the plates are assumed thin and as such disturbances can travel at arbitrarily high speeds and (b) one cannot immediately rule out the existence of instabilities that grow exponentially in time. In this paper, a new unified approach to boundary value problems is employed \cite{fokas2008uab}. This approach is direct, and removes the need to invoke the Laplace transform.

We choose to work with the more physically relevant case in which the vertical extent of the channel is finite and avoid the use of the Laplace tranform. The problem is reformulated in terms of a spectral parameter $k$, and in the spectral space the evolution equations for the amplitutdes of the plates take a simplified form.

The problem we study is non-dimensionalised, with the only remaining parameter being the speed of the horizontal flow between each plate. Referring to \cite{crighton1991flm} for details of the re-scaling scheme, the dimensionaless parameter $U$ arises in the form:
\[ U = \left( \frac{ m_* ^{3/2}}{\rho_* B_*^{1/2}}\right) U_*. \]
Here $B_*$ is the bending stiffness of the plate, $m_*$ is the mass per unit length of the plate, $\rho_*$ is the density of the fluid and $U_*$ is the horizontal velocity of the flow. In practice, the parameter $U$ will be small and has served as a basis for asymptotic approaches to the problem in many papers. For instance, in \cite{peake2001nonlinear} the author gives the following example: if the plate is made of $2$cm steel, and the fluid is water, moving at around $10m s^{-1}$ (which can serve as an approximate upper limit), then $U\sim 0.05 \ll 1$.

We first work on a formal level to derive explicit solutions to the underlying problem, \emph{then} prove what properties these solutions have. Frequently we will refer to the standard Sobolev space $H^s_{\dd \nu}(\R^m)=W^{s,2}_{\dd \nu}(\R^m)$, where:
\[ W^{k,p}_{\dd  \nu}(\R^m) \defn \{ \partial^\alpha f\in  L^p_{\dd \nu}(\R^m), 0\leq |\alpha |\leq k \}.\] 
and $L^p_{\dd \nu}(\R^m)$ is the usual space of Lebesgue integrable functions on $\R^m$ with norm
\[ \left( \int |f|^p\, \dd \nu\right)^{1/p}.\]
The space of $k$ times continuously differentiable functions from $\mathbf{R}^m$ to $\mathbf{R}$ will be denoted $C^k(\mathbf{R}^m)$, and the Schwartz space $\mathcal{S}(\mathbf{R}^m)$ is defined to consist of smooth functions $f$ from $\mathbf{R}^m$ to $\mathbf{R}$ such that $\sup_x |x^\alpha \partial^\beta f|<\infty$ for all multi-indices $\alpha$, $\beta$.

\subsection{The Governing Equations}
The problem we study concerns the motion of a collection of $n$ elastic plates, lying horizontally, each driven by a mean flow with horizontal velocity $U_i$. We assume each plate can be described by a single valued function $\eta_i (x,t)$ and we denote the corresponding surface by $\Gamma_i$:
\begin{equation} \Gamma_i = \{ (x,y)\in \R ^2 : y=\eta_i (x,t)\}, \qquad t> 0 .\end{equation}
with upward normal $N(\Gamma_i)=(-\partial_x\eta_i,1)$. The plates are situated in a horiztonal channel $\Omega$, bound above and below by the surfacaes $\mathcal{B}^+$ and $\mathcal{B}^-$:
\begin{equation} \mathcal{B}^\pm = \{ (x,y) \in\R ^2: y= \pm h_0+ h_\pm(x)\},  \end{equation}
where $h_0>0$ and $h_\pm \in \mathcal{S}(\R )$. We define $\Omega_i$, $1\leq i \leq n-1$, to be the open, connected region bounds between the surfaces $\Gamma_i$ and $\Gamma_{i+1}$ so that:
\begin{equation} \Omega_i = \{ (x,y) \in \R ^2: \eta_i(x,t) < y < \eta_{i+1}(x,t)\}, \qquad t> 0. \end{equation}
The connectedness condition means that for $i>j$, we have $\eta_i(x,t)>\eta_j(x,t)$. In addition, we define the distinguished regions:
\begin{align}
\Omega_0 &= \{(x,y)\in\R ^2: -h_0 +h_-(x)<y<\eta_1(x,t)\}, \qquad t> 0, \\
\Omega_n &= \{(x,y)\in\R ^2: \eta_n(x,t)<y<+h_0 +h_+(x)\}, \qquad t> 0.
\end{align}
With these definitions it is clear that $\Omega = \cup_{i=0}^n \overline\Omega_{i}$. A diagram for the geometry of the problem is given in Figure \ref{diagram}.
\begin{figure}\label{diagram}
\begin{center}
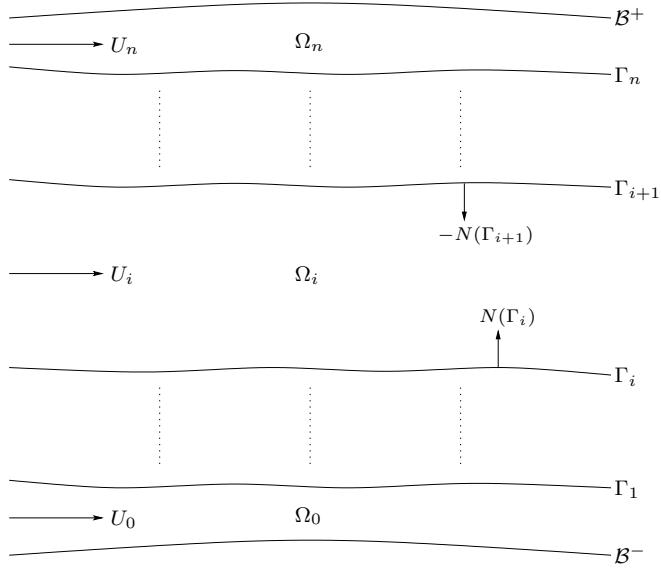
\end{center}
\caption{The geometry of the problem.}
\end{figure}

Each $\Omega_i$ is occupied by an incompressible, irrotational fluid travelling with mean velocity $U_i$ in $x$-direction. We let $\phi_i(x,y,t)$ denote the velocity potential for the perturbation from the mean flow in $\Omega_i$, and as such:
\begin{equation} \Delta \phi_i =0 \qquad \textrm{in $\Omega_i$}, \end{equation}
which follows from the incompressibility condition. The evolution of the surface $\Gamma_i$ is determined by beam equation:
\begin{equation} \partial_t^2\eta_i + \partial_x^4 \eta_i = p_{i-1} - p_i \qquad \textrm{on $\Gamma_i$}, \label{beam_pair} \end{equation}
where $p_i=p_i(x,y,t)$ is the pressure of the fluid in $\Omega_i$. In addition, on the upper and lower sides of $\Gamma_i$ we have the Bernoulli conditions:
\begin{equation} \left. \begin{array}{rcl}
\partial_{U_{i}}\phi_{i} + \tfrac{1}{2} \|\nabla\!  \phi_i\|^2  &=& -p_i \\
\partial_{U_{i-1}}\phi_{i-1} + \tfrac{1}{2}\|\nabla\! \phi_{i-1}\|^2  &=& -p_{i-1} \end{array}\right\}\quad \textrm{on $\Gamma_i$}. \label{Bernoulli_pair} \end{equation}
where we have introduced the differential operator $\partial_U \equiv \partial_t + U\partial_x$ for convenience of notation. Using the pair of Bernoulli conditions in \eqref{Bernoulli_pair}, we can eliminate the pressure terms in \eqref{beam_pair} to find the single boundary condition on the surface $\Gamma_i$ in terms of the functions $\{\eta_i, \phi_i, \phi_{i-1}\}$:
\begin{equation} \partial_t^2\eta_i + \partial_x^4 \eta_i + \partial_{U_{i-1}}\phi_{i-1} - \partial_{U_i}\phi_i 
   + \tfrac{1}{2}\|\nabla\! \phi_{i-1}\|^2- \tfrac{1}{2} \|\nabla\!  \phi_i\|^2  =0 \qquad \textrm{on $\Gamma_i$.} \label{bernoullitotal}
\end{equation}

We also have two kinematic conditions on each $\Gamma_i$, since fluid particles initially on the surface must remain on the surface. 
\begin{equation}\label{kinematic_pair}
\left.\begin{array}{rcl}(U_{i-1} + \partial_x \phi_{i-1}, \partial_y\phi_{i-1}) \cdot N(\Gamma_i)\!\! &=&\!\! \partial_t \eta_i \\
\Big((U_{i-1} + \partial_x \phi_{i-1}, \partial_y\phi_{i-1})-(U_i + \partial_x \phi_i, \partial_y\phi_i)\Big) \cdot N(\Gamma_i) \!\! &=&\!\! 0 \end{array}\right\}\quad \textrm{on $\Gamma_i$}
\end{equation}
To summarise, the nonlinear hydrodynamic problem studied in this paper is governed by the following $n+1$ dynamic boundary value problems:
\subsubsection*{The internal domains $\Omega_i$, $1\leq i \leq n-1$:}
\vspace{-0.5cm}
\begin{subequations}\label{alleqns}
\begin{alignat}{2}\Delta\phi_i &= 0&\quad &\textrm{in $\Omega_i$,} \label{harmonic}\\
(U_i+\partial_x\phi_i, \partial_y\phi_i)\cdot N(\Gamma_i) &= \partial_t\eta_i &\quad &\textrm{on $\Gamma_i$,} \label{dynamicdown} \\
\partial_t^2\eta_i + \partial_x^4 \eta_i + \partial_{U_{i-1}}\phi_{i-1} + \tfrac{1}{2}\|\nabla\! \phi_{i-1}\|^2   &=\partial_{U_i}\phi_i+\tfrac{1}{2} \|\nabla\!  \phi_i\|^2&\quad &\textrm{on $\Gamma_i$,} \label{beamdown} \\
(U_i+\partial_x\phi_i, \partial_y\phi_i)\cdot N(\Gamma_{i+1}) &= \partial_t\eta_{i+1} &\quad &\textrm{on $\Gamma_{i+1}$,} \label{dynamicup} \\
\partial_t^2\eta_{i+1} + \partial_x^4 \eta_{i+1} + \partial_{U_{i}}\phi_{i} + \tfrac{1}{2}\|\nabla\! \phi_{i}\|^2   &=\partial_{U_{i+1}}\phi_{i+1}+\tfrac{1}{2} \|\nabla\!  \phi_{i+1}\|^2&\quad &\textrm{on $\Gamma_{i+1}$.} \label{beamup}
\end{alignat}
\end{subequations}
The fact that \emph{two} conditions are given on the upper and lower parts of $\partial\Omega_i$ is a consequence of the fact that the surfaces $\Gamma_i$ are unknown.
\subsubsection*{The top domain $\Omega_n$:}
\vspace{-0.5cm}
\begin{subequations}\label{alleqnsup}
\begin{alignat}{2}\Delta\phi_n &= 0&\quad &\textrm{in $\Omega_n$,} \label{harmonic-n}\\
(U_n+\partial_x\phi_n, \partial_y\phi_n)\cdot N(\Gamma_n) &= \partial_t\eta_n &\quad &\textrm{on $\Gamma_n$,} \label{dynamicdown-n} \\
\partial_t^2\eta_n + \partial_x^4 \eta_n + \partial_{U_{n-1}}\phi_{n-1} + \tfrac{1}{2}\|\nabla\! \phi_{n-1}\|^2   &=\partial_{U_n}\phi_i+\tfrac{1}{2} \|\nabla\!  \phi_n\|^2&\quad &\textrm{on $\Gamma_n$,} \label{beamdown-n} \\
(U_i+\partial_x\phi_i, \partial_y\phi_i)\cdot N(\mathcal{B}^+) &=0 &\quad &\textrm{on $\mathcal{B}^+$.} \label{neumann-n} 
\end{alignat}
\end{subequations}
Here only one boundary condition is given on $\mathcal{B}^-$, since this boundary is fixed and determined in terms of the given funciton $h_+(x)$. This condition is the usual Neumann boundary condition reflecting the fact that the fluid cannot escape through $\mathcal{B}^-$
\subsubsection*{The bottom domain $\Omega_0$:}
\vspace{-0.5cm}
\begin{subequations}\label{alleqnsdown}
\begin{alignat}{2}\Delta\phi_0 &= 0&\quad &\textrm{in $\Omega_0$,} \label{harmonic-0}\\
(U_0+\partial_x\phi_0, \partial_y\phi_0)\cdot N(\Gamma_1) &= \partial_t\eta_1 &\quad &\textrm{on $\Gamma_1$,} \label{dynamicdown-0} \\
\partial_t^2\eta_1 + \partial_x^4 \eta_1 + \partial_{U_{0}}\phi_{0} + \tfrac{1}{2}\|\nabla\! \phi_{0}\|^2   &=\partial_{U_1}\phi_1+\tfrac{1}{2} \|\nabla\!  \phi_1\|^2&\quad &\textrm{on $\Gamma_1$,} \label{beamup-0} \\
(U_0+\partial_x\phi_0, \partial_y\phi_0)\cdot N(\mathcal{B}^-) &=0 &\quad &\textrm{on $\mathcal{B}^-$.} \label{neumann-0} 
\end{alignat}
\end{subequations}
Again, only one boundary condition (Neumann) is given on the fixed surface $\mathcal{B}^-$, since this surface is given in terms of the known funciton $h_-(x)$.
\vspace{5mm}

In addition to each system of equations, we require that the amplitudes decay to zero at large $x$, and similarly we assume $\phi$ and its derivatives decay to zero for large $x$, uniformly in $y$. With this assumption the ``effective'' part of each $\partial\Omega_i$ is $\Gamma_i$ (bottom) and $\Gamma_{i+1}$ (top), since all the functions are zero on $\partial\Omega_i \setminus (\Gamma_i\cup \Gamma_{i+1})$.

\subsection{The New Dependent Coordinates}
In each of the boundary value problems listed above, it is clear that whatever the harmonic function $\phi_i$ is, it will be determined by the values of it and its derivatives on the boundary $\partial\Omega_i$. It is convenient then, to reformulate the problem in terms of the the potentials \emph{evaluated on} $\Gamma_i$. We do this by introducing the new functions $\xi^\pm_i(x,t)$:
\begin{subequations}
\begin{align}
\xi^+_i(x,t) &= \phi_i|_{\Gamma_{i+1}}, \\
\xi^-_i(x,t) &= \phi_i|_{\Gamma_i},
\end{align}
\end{subequations}
where $\phi|_{\Gamma_i}=\phi(x,\eta_i(x,t),t)$ etc. So $\xi^+(x,t)$ is the potential evaluated on the top of the domain $\Omega_i$ (i.e. $\Gamma_{i+1}$), and $\xi^-(x,t)$ is the potential evaluated on the bottom of the domain $\Omega_i$ (i.e. $\Gamma_{i}$). The chain rule gives rise to the following on $\Gamma_{i+1}$:
\begin{subequations}\label{chainrules(i+1)}
\begin{align}
\partial_x \xi^+_i &= (\partial_x \phi_i + \partial_x \eta_{i+1} \partial_y \phi_i)|_{\Gamma_{i+1}}, \\
\partial_t \xi^+_i &= (\partial_t \phi_i + \partial_t \eta_{i+1} \partial_y \phi_i)|_{\Gamma_{i+1}}, 
\end{align}
\end{subequations}
and similarly on $\Gamma_i$:
\begin{subequations}\label{chainrules(i)}
\begin{align}
\partial_x \xi^-_i &= (\partial_x \phi_i + \partial_x \eta_{i} \partial_y \phi_i)|_{\Gamma_i}, \label{chainrules(i)a}\\
\partial_t \xi^-_i &= (\partial_t \phi_i + \partial_t \eta_{i} \partial_y \phi_i)|_{\Gamma_i}.
\end{align}
\end{subequations}
It is possible to express all the derivatives of $\phi_i$ on the boundary $\partial\Omega_i$ in terms of these new functions. For example, \eqref{kinematic_pair} and \eqref{chainrules(i)} yield the following non-singular set of equations for the derivatives of $\phi_i$ on $\Gamma_i$:
\begin{equation}
\begin{bmatrix} -\partial_x\eta_i & 1 & 0 \\ 0 & \partial_t\eta_i & 1 \\ 1 & \partial_x\eta_i & 0 \end{bmatrix} \begin{bmatrix} \partial_x\phi_i \\ \partial_y \phi_i \\ \partial_t \phi_i \end{bmatrix} = \begin{bmatrix} \partial_{U_i}\eta_i \\ \partial_t \xi_i^- \\ \partial_x \xi_i^- \end{bmatrix} \qquad \textrm{on $\Gamma_i$.}
\end{equation}
A similar set of equations hold for the derivatives of $\phi_i$ on $\Gamma_{i+1}$. Solving these equations and using the notation $\langle X\rangle \equiv (1+X^2)^{1/2}$ gives:
\begin{equation}
\begin{bmatrix} \partial_x\phi_i \\ \partial_y \phi_i \\ \partial_t \phi_i \end{bmatrix} = \frac{1}{\langle \partial_x\eta_i\rangle ^2} \begin{bmatrix} \partial_x \xi^-_i - \partial_x \eta_i\partial_{U_i}\eta_i \\ \partial_x \xi^-_i \partial_x \eta_i + \partial_{U_i}\eta_i \\ \langle \partial_x \eta_i\rangle^2 \partial_t\xi^-_i - \partial_t\eta_i( \partial_{U_i}\eta_i +  \partial_x \eta_i \partial_x \xi^-_i) \end{bmatrix} \quad \textrm{on $\Gamma_i$,} \label{mess}
\end{equation}
and
\begin{equation}
\begin{bmatrix} \partial_x\phi_i \\ \partial_y \phi_i \\ \partial_t \phi_i \end{bmatrix} = \frac{1}{\langle \partial_x\eta_{i+1}\rangle ^2} \begin{bmatrix} \partial_x \xi^+_i - \partial_x \eta_{i+1}\partial_{U_i}\eta_{i+1} \\ \partial_x \xi^+_i \partial_x \eta_{i+1} + \partial_{U_i}\eta_{i+1} \\ \langle \partial_x \eta_{i+1}\rangle^2 \partial_t\xi^+_i - \partial_t\eta_{i+1}( \partial_{U_i}\eta_{i+1} + \partial_x \eta_{i+1} \partial_x \xi^+_i) \end{bmatrix} \quad \textrm{on $\Gamma_{i+1}$.} \label{mess2}
\end{equation}
These expressions indicate that it is enough to know the functions $\xi^\pm_i(x,t)$ to reproduce the potential $\phi_i(x,y,t)$ on $\Omega_i$. The problem now consists of determining each amplitutde $\eta_i(x,t)$ and the potentials each $\partial\Omega_i$. The total number of unknowns is:
\[ \textrm{(\# of amplitudes)} + \textrm{(\# of potentials on each $\Gamma_i$)}+ \textrm{(potentials on $\mathcal{B}^\pm$)} = 3n+2.  \]
Our aim is now to reduce our complex boundary value problem into a system of $3n+2$ equations for the unknowns. The Bernoulli condition \eqref{bernoullitotal} on each $\Gamma_i$ will give $n$ equations in terms of the $\{\xi_i^\pm\}$ and $\{\eta_i\}$, so we are left to produce $2(n+1)$ independent equations, i.e. two equations from each $\Omega_i$, so that the problem is well posed.

\section{The Non-Local Formulation}
In this section we present a non-local formulation of the problem that involves translating the harmonicity of each $\phi_i$, along with the dynamic (or Neumann) boundary conditions on $\partial\Omega_i$, into a pair of integro-differential equations in $\xi^{\pm}$ and $\eta_i$. This is achieved by constructing a so-called \emph{global relation} \cite{fokas1997utm,fokas2008uab,ashton2007fkf}. The global relation is a consequence of the following simple lemma.
\begin{lemma}\label{global_lem}
Suppose that the functions $u$ and $v$ are harmonic in $\Omega\subset \R ^2$. Then the following holds:
\begin{equation}
\partial_x \left(\partial_y u \, \partial_x v+ \partial_y v \, \partial_x u\right) + \partial_y \left( \partial_y u \, \partial_y v - \partial_x u \, \partial_x v\right)  = 0 \label{global_lem_eqn}
\end{equation}
for each $(x,y) \in \Omega$.
\end{lemma}
\noindent
This can be verified by expanding out the left hand side of \eqref{global_lem_eqn} to find:
\[ (\partial_y v)\Delta u+ (\partial_y u) \Delta v \]
which vanishes in $\Omega$ if both $u$ and $v$ are harmonic in $\Omega$.

In what follows we shall be integrating expressions such as \eqref{global_lem_eqn} over each $\Omega_i$, and as such we need conditions on the functions $\{\eta_i,\phi_i\}$ and their derivatives to decay sufficiently rapidly as $|x|\rightarrow \infty$. At this stage we make the assumption, a priori, that the functions have the necessary conditions so that the integrals converge. It is possible to make these conditions precise using standard Sobolev estimates, but it is preferable at this stage to work on a formal level.

Suppose $k\in\R $ and define $v \in C^\infty(\R ^2)$ by:
\[ v(x,y)=\exp (-\ii kx +\kappa y), \qquad \kappa = \pm k. \]
Clearly $v$ is harmonic and bounded in each $\Omega_i$. It follows from an application of lemma \ref{global_lem} that the following holds in each $\Omega_i$:
\begin{equation} \partial_x \Big( e^{-\ii kx + \kappa y} (\kappa  \partial_x \phi_i-\ii k  \partial_y\phi_i )\Big) + \partial_y\Big( e^{-\ii kx + \kappa y}(\kappa  \partial_y\phi_i + \ii k \partial_x\phi_i)\Big)  = 0.\label{totdiv}\end{equation}
We now look at internal, top and bottom domains in $\Omega$ seperately.
\subsubsection*{The internal domain $\Omega_i$, $1\leq i \leq n-1$:}
We integrate \eqref{totdiv} about an internal domain $\Omega_i$. An application of the divergence theorem gives:
\begin{multline} \int_{\Gamma_i} e^{-\ii k x + \kappa y}(\kappa  \partial_x \phi_i-\ii k  \partial_y\phi_i ,\kappa  \partial_y\phi_i + \ii k \partial_x\phi_i)\cdot N(\Gamma_i)\, \dd x \\
- \int_{\Gamma_{i+1}}  e^{-\ii k x + \kappa y}(\kappa  \partial_x \phi_i-\ii k  \partial_y\phi_i ,\kappa  \partial_y\phi_i + \ii k \partial_x\phi_i)\cdot N(\Gamma_{i+1})\, \dd x =0, \label{divint}
\end{multline}
where we have discarded the contributions from $|x|\rightarrow \infty$, since the fields are assumed to vanish there. The contribution in \eqref{divint} from $\Gamma_i$ is given by:
\[ \int e^{-\ii k x + \kappa \eta_i}\Big(\kappa  \left( \partial_y\phi_i - \partial_x\eta_i \partial_x\phi_i  \right)|_{\Gamma_i} + \ii k \left( \partial_x\phi_i + \partial_x\eta_i \partial_y\phi_i\right)|_{\Gamma_i}\Big)\, \dd x. \]
We see that the first term in the integrand is $\partial_{U_i}\eta_i$ by \eqref{dynamicdown}, and the second is $\partial_x\xi^-_i$ by \eqref{chainrules(i)}. The expression becomes:
\begin{equation} \int e^{-\ii k x + \kappa \eta_i}\Big(\kappa  \partial_{U_i}\eta_i  + \ii k \partial_x\xi_i^- \Big)\, \dd x. \label{intcontrib(i)} \end{equation}
Performing a similar calculation, we find the contribution from $\Gamma_{i+1}$ to be:
\begin{equation} -\int e^{-\ii k x + \kappa \eta_{i+1}}\Big(\kappa  \partial_{U_i}\eta_{i+1} + \ii k \partial_x\xi_i^+ \Big)\, \dd x. \label{intcontrib(i+1)} \end{equation}
Combining equations \eqref{intcontrib(i)}, \eqref{intcontrib(i+1)} and \eqref{divint}, we find the following 1-parameter family of integrodifferential equations:
\begin{equation} \int e^{-\ii k x}\Big(e^{\kappa\eta_i} \left(\kappa  \partial_{U_i}\eta_i  + \ii k \partial_x\xi_i^- \right)- e^{\kappa \eta_{i+1}}\left(\kappa  \partial_{U_i}\eta_{i+1} + \ii k \partial_x\xi_i^+\right) \Big)\, \dd x=0. \label{globalint} \end{equation}
Recalling that $\kappa=\pm k$, \eqref{globalint} gives \emph{two} equations, which can be added and subtracted to give the following result.
\begin{proposition}\label{prop1}
The incomplete boundary value problem in \eqref{harmonic}, \eqref{dynamicdown} and \eqref{dynamicup} is equivalent to the the pair of integro-differential equations:
\begin{subequations}\label{integrals}
\begin{multline}
 \int e^{-\ii kx} \Big( \partial_{U_i}\eta_i \sinh (k\eta_i) + \ii \partial_x \xi_i^- \cosh (k\eta_i) \\ - \partial_{U_i} \eta_{i+1} \sinh(k\eta_{i+1}) - \ii \partial_x\xi_i^+\cosh(k\eta_{i+1}) \Big)\dd x =0, \label{int1}\end{multline}
and
\begin{multline}
 \int e^{-\ii kx} \Big( \partial_{U_i}\eta_i \cosh (k\eta_i) + \ii \partial_x \xi_i^- \sinh (k\eta_i) \\ - \partial_{U_i} \eta_{i+1} \cosh(k\eta_{i+1}) - \ii \partial_x\xi_i^+\sinh(k\eta_{i+1}) \Big)\dd x =0, \label{int2}
\end{multline}
which are valid for $k\in\R $. 
\end{subequations}
\end{proposition}
\begin{remark}
We note that if both $\eta_{i+1}$ and $\eta_i$ are small, in an appropriate sense, then \eqref{int1} and the completeness of the Fourier transform yield $\partial_x\xi^+ \sim \partial_x\xi^-$. This is consistent with what one would expect, because over short length scales the speed of the flow in the $x$-direction would remain approximately constant in the absense of viscosity.
\end{remark}
\subsubsection*{The top domain $\Omega_n$:}
Integrating \eqref{totdiv} over $\Omega_n$ and using the divergence theorem we find:
\begin{multline} \int_{\Gamma_n} e^{-\ii k x + \kappa y}(\kappa  \partial_x \phi_i-\ii k  \partial_y\phi_i ,\kappa  \partial_y\phi_i + \ii k \partial_x\phi_i)\cdot N(\Gamma_n)\, \dd x \\
- \int_{\mathcal{B}^+}  e^{-\ii k x + \kappa y}(\kappa  \partial_x \phi_i-\ii k  \partial_y\phi_i ,\kappa  \partial_y\phi_i + \ii k \partial_x\phi_i)\cdot N(\mathcal{B}^+)\, \dd x =0. \label{divint_top}
\end{multline}
In analogy with \eqref{intcontrib(i)} we find the contribution from $\Gamma_n$ to be:
\begin{equation} \int e^{-\ii k x + \kappa \eta_n}\Big(\kappa  \partial_{U_n}\eta_n  + \ii k \partial_x\xi_n^- \Big)\, \dd x. \label{intcontrib(n)}\end{equation}
The contribution from $\mathcal{B}^+$, however, takes a simpler form owing to the fact that is boundary is fixed. Using the Neumann condition \eqref{neumann-n} with $N(\mathcal{B}^+)=(-\partial_xh^+,1)$ and the definition $\xi^+_n(x,t)=\phi(x,h_0+h_+(x),t)$ we find the contribution from $\mathcal{B}^+$ to be:
\begin{equation} -\ii k \int e^{-\ii k x + \kappa (h_0+h_+) } \partial_x \xi^+_n\, \dd x .\label{intcontrib(n+1)}\end{equation}
Combining \eqref{intcontrib(n)} and \eqref{intcontrib(n+1)} we find the following 1-parameter family of integro-differential equations:
\[ \int e^{-\ii k x}\Big(e^{\kappa\eta_n} \left(\kappa  \partial_{U_n}\eta_n  + \ii k \partial_x\xi_n^- \right)- \ii k e^{\kappa (h_0+h_+)}  \partial_x\xi_n^+ \Big)\, \dd x=0. \]
In analogy with Proposition \ref{prop1} we arrive at the following.
\begin{proposition}\label{prop2}
The incomplete boundary value problem in \eqref{harmonic-n}, \eqref{dynamicdown-n} and \eqref{neumann-n} is equivalent to the the pair of integro-differential equations:
\begin{subequations}\label{integrals-n}
\begin{equation}
 \int e^{-\ii kx} \Big( \partial_{U_n}\eta_n \sinh (k\eta_n) + \ii \partial_x \xi_n^- \cosh (k\eta_n) - \ii \partial_x\xi_n^+\cosh(k(h_0+h_+)) \Big)\dd x =0, \label{int1-n}\end{equation}
and
\begin{equation}
 \int e^{-\ii kx} \Big( \partial_{U_n}\eta_n \cosh (k\eta_n) + \ii \partial_x \xi_n^- \sinh (k\eta_n) - \ii \partial_x\xi_n^+\sinh(k(h_0+h_+)) \Big)\dd x =0, \label{int2-n}\end{equation}
which are valid for $k\in\R $. 
\end{subequations}
\end{proposition}
\begin{remark}\label{flat_rem1}
We see that if $\mathcal{B}^+$ is flat, i.e. $h_+\equiv 0$, then we can eliminate $\xi^+_n(x,t)$ from equations \eqref{integrals-n}. Multiplying \eqref{int1-n} by $\sinh (kh_0)$, \eqref{int2-n} by $\cosh (kh_0)$ and subtracting yields the single global relation:
\begin{equation}
\int e^{-\ii k x} \Big( \partial_{U_n} \eta_n \cosh (k(\eta_n - h_0)) + \ii k \partial_x \xi^-_n \sinh (k(\eta_n-h_0)) \Big)\dd x =0
\end{equation}
\end{remark}

\subsubsection*{The bottom domain $\Omega_0$:}
The derivation of the results for the bottom domain is analogous to the previous calculation for the top domain.
\begin{proposition}\label{prop3}
The incomplete boundary value problem in \eqref{harmonic-0}, \eqref{dynamicdown-0} and \eqref{neumann-0} is equivalent to the the pair of integro-differential equations:
\begin{subequations}\label{integrals-0}
\begin{equation}
 \int e^{-\ii kx} \Big( \partial_{U_0}\eta_1 \sinh (k\eta_1) + \ii \partial_x \xi_0^+ \cosh (k\eta_1) - \ii \partial_x\xi_0^-\cosh(k(h_0-h_-)) \Big)\dd x =0, \label{int1-0}\end{equation}
and
\begin{equation}
 \int e^{-\ii kx} \Big( \partial_{U_0}\eta_1 \cosh (k\eta_1) + \ii \partial_x \xi_0^+ \sinh (k\eta_1) + \ii \partial_x\xi_0^-\sinh(k(h_0-h_-)) \Big)\dd x =0, \label{int2-0}\end{equation}
which are valid for $k\in\R $. 
\end{subequations}
\end{proposition}
\begin{remark}\label{flat_rem2}
In analogy with the result in Remark \ref{flat_rem1}, we note that in the case $\mathcal{B}_-$ flat, we can eliminate $\xi^-_0(x,t)$ from equations \eqref{integrals-0}. Multiplying \eqref{int1-0} by $\sinh (kh_0)$, \eqref{int2-0} by $\cosh (kh_0)$ and adding yields the single global relation:
\begin{equation}
\int e^{-\ii k x} \Big( \partial_{U_0} \eta_1 \cosh (k(\eta_1 + h_0)) + \ii k \partial_x \xi^+_0 \sinh (k(\eta_1+h_0)) \Big)\dd x =0
\end{equation}
\end{remark}

The non-local formulation of the problem has given two \emph{explicit} equations in each $\Omega_i$, $0\leq i\leq n$ for the functions $\{\xi_i^\pm\}$. Together with the $n$ Bernoulli-type conditions, we have:
\[ 2n+2 + n \equiv 3n+2 \]
equations for the $3n+2$ unknowns, so our problem is well-posed in this sense. The important feature of our $3n+2$ equations is that they are \emph{explicit} in the functions $\{\eta_i, \xi_i^\pm\}$, which makes analysis of them much simpler. For instance, they allow direct use of perturbation expansions, direct application of standard Sobolev estimates and direct application of standard asymptotic techniques -- we refer to \cite{ashton2008nfr} for examples of such techniques used in a related non-local formulation for rotational water waves.

We end this section with the following result, which summarises the non-local formulation for the dynamic boundary value problem outlined in \eqref{alleqns}, \eqref{alleqnsup} and \eqref{alleqnsdown} along with the $n$ bernoulli-type condtions on the $\{\Gamma_i\}_{i=1}^n$ in terms of $\{\xi_i^\pm\}_{i=0}^n$.
\begin{proposition}\label{summary_prop}
The solution to the boundary value problem in \eqref{alleqns}, \eqref{alleqnsup} and \eqref{alleqnsdown} is completely determined by the $3n+2$ functions $\{\eta_i\}_{i=1}^n$, $\{\xi_i^\pm\}_{i=0}^n$ which satisfy:
\begin{itemize}
 \item The $2n-2$ non-local \textbf{internal equations}:
\begin{multline}
 \int e^{-\ii kx} \Big( \partial_{U_i}\eta_i \sinh (k\eta_i) + \ii \partial_x \xi_i^- \cosh (k\eta_i) \\ - \partial_{U_i} \eta_{i+1} \sinh(k\eta_{i+1}) - \ii \partial_x\xi_i^+\cosh(k\eta_{i+1}) \Big)\dd x =0, \label{summaryint1}\end{multline}
\begin{multline}
 \int e^{-\ii kx} \Big( \partial_{U_i}\eta_i \cosh (k\eta_i) + \ii \partial_x \xi_i^- \sinh (k\eta_i) \\ - \partial_{U_i} \eta_{i+1} \cosh(k\eta_{i+1}) - \ii \partial_x\xi_i^+\sinh(k\eta_{i+1}) \Big)\dd x =0, \label{summaryint2}
\end{multline}
which are valid for $k\in \R $ and $1\leq i \leq n-1$.
 \item The $4$ non-local \textbf{external equations}:
\begin{subequations}\label{summaryintegrals}
\begin{equation}
 \int e^{-\ii kx} \Big( \partial_{U_n}\eta_n \sinh (k\eta_n) + \ii \partial_x \xi_n^- \cosh (k\eta_n) - \ii \partial_x\xi_n^+\cosh(k(h_0+h_+)) \Big)\dd x =0, \label{summaryint1-n}\end{equation}
\begin{equation}
 \int e^{-\ii kx} \Big( \partial_{U_n}\eta_n \cosh (k\eta_n) + \ii \partial_x \xi_n^- \sinh (k\eta_n) - \ii \partial_x\xi_n^+\sinh(k(h_0+h_+)) \Big)\dd x =0, \label{summaryint2-n}\end{equation}
\begin{equation}
 \int e^{-\ii kx} \Big( \partial_{U_0}\eta_1 \sinh (k\eta_1) + \ii \partial_x \xi_0^+ \cosh (k\eta_1) - \ii \partial_x\xi_0^-\cosh(k(h_0-h_-)) \Big)\dd x =0, \label{summaryint1-0}\end{equation}
\begin{equation}
 \int e^{-\ii kx} \Big( \partial_{U_0}\eta_1 \cosh (k\eta_1) + \ii \partial_x \xi_0^+ \sinh (k\eta_1) + \ii \partial_x\xi_0^-\sinh(k(h_0-h_-)) \Big)\dd x =0, \label{summaryint2-0}\end{equation}
which are valid for $k\in\R $. 
\end{subequations}
 \item The $n$ \textbf{Bernoulli-type equations}:
\begin{equation} \partial_t^2\eta_i + \partial_x^4 \eta_i + \partial_{U_{i-1}}\phi_{i-1} + \tfrac{1}{2}\|\nabla\! \phi_{i-1}\|^2   =\partial_{U_i}\phi_i+\tfrac{1}{2} \|\nabla\!  \phi_i\|^2 \qquad \textrm{on $\Gamma_i$}\label{summarybernoulli} \end{equation}
for $1\leq i \leq n$, where $\nabla \phi_i|_{\Gamma_i}$ and $\partial_t \phi_i |_{\Gamma_i}$ etc. are given in terms of $\xi_i^\pm$ via \eqref{mess} and \eqref{mess2}.
\end{itemize}
\end{proposition}

\section{The Linear Problem and Well-Posedness}
In this section we consider the linearisation of the equations appearing Proposition \ref{summary_prop} and the corresponding Cauchy problem for the linearised equations. We denote $\eta_i = h_i + \tilde{\eta}_i$ where the $\{h_i\}$ are constant and $\|\partial_x\xi_i^\pm\|_{L^\infty}$ and $\|\tilde{\eta}_i\|_{L^\infty}$ are small so that quadratic terms and beyond may be neglected. For simplicity we restrict attention to the case in which $\mathcal{B}^\pm$ are flat. It is straight forward to adjust the analysis to deal with the case in which the upper and lower parts of $\partial\Omega$ vary according to a funciton in the Schwartz class.

We use $\hat{f}$ to denote the Fourier transform of the function $f$, and drop the tilde from $\tilde{\eta_i}$ for notational convenience. It will also be useful to denote the upper and lower parts of the channel by $y=h_{n+1}$ (top) and $y=h_0$ (bottom). Discarding quadratic terms and higher, the linearised equations are written as:
\begin{itemize}
 \item The $2n-2$ \textbf{internal equations}:
\begin{subequations}\label{interlin}
\begin{multline} (\partial_t \hat{\eta}_i + \ii k U_i \hat{\eta}_i )\sinh (kh_i)- (\partial_t \hat{\eta}_{i+1} + \ii k U_i \hat{\eta}_{i+1})\sinh (kh_{i+1}) \\  + k\hat{\xi}_i^+ \cosh (kh_{i+1}) - k\hat{\xi}_i^-  \cosh (kh_i) =0, \end{multline}
\begin{multline} (\partial_t \hat{\eta}_i + \ii k U_i \hat{\eta}_i )\cosh (kh_i)- (\partial_t \hat{\eta}_{i+1} + \ii k U_i \hat{\eta}_{i+1})\cosh (kh_{i+1}) \\  + k\hat{\xi}_i^+ \sinh (kh_{i+1}) - k\hat{\xi}_i^-  \sinh (kh_i) =0, \end{multline}
\end{subequations}
for $1\leq i \leq n$.
 \item The $2$ \textbf{external equations}:
\begin{subequations}\label{exterlin}
\begin{align} (\partial_t \hat{\eta}_n + \ii k U_n \hat{\eta}_n) \cosh (k(h_{n+1} - h_n)) + k \hat{\xi}^-_n \sinh (k(h_{n+1} - h_n)) &=0, \\
 (\partial_t \hat{\eta}_1 + \ii k U_0 \hat{\eta}_0) \cosh (k(h_1 - h_0)) + k \hat{\xi}_0^+ \sinh(k(h_1 -h_0)) &=0.
\end{align}
\end{subequations}
 \item The $n$ \textbf{Bernoulli-type conditions}:
\begin{equation} \partial^2_t \hat{\eta}_i + k^4 \hat{\eta}_i + (\partial_t \hat{\xi}^+_{i-1} + \ii k U_{i-1} \hat{\xi}^+_{i-1}) = (\partial_t \hat{\xi}^-_i + \ii k U_i \hat{\xi}^-_i), \label{bernoullilin}\end{equation}
for $1\leq i \leq n$.
\end{itemize}
Note that we have eliminated $\xi^+_n$ and $\xi^-_0$ using two of the external equations (see Remarks \ref{flat_rem1} and \ref{flat_rem2}) leaving $3n$ equations for $3n$ unknowns. Equations \eqref{interlin} and \eqref{exterlin} allow us to express the $\{\hat{\xi}_i^\pm\}$ in terms of $\{ \hat{\eta}_i\}$, which can then be used in \eqref{bernoullilin} to determine the governing equation for the $\{\hat{\eta}_i\}$. Introducing the functions $\varphi (k,t) = (\hat{\eta}_1, \ldots , \hat{\eta}_n)^T$, $\pi(k,t)= (\partial_t \hat{\eta}_1, \ldots , \partial_t \hat{\eta}_n)^T$ and $\Delta h_i = h_{i+1}-h_i$, the governing equations are:
\begin{align*} & \Big[ - \tfrac{\mathrm{csch}(k \Delta h_i)}{k} \Big] \dot{\pi}_{i+1} + \Big[ 1 + \tfrac{\coth (k\Delta h_i)}{k} + \tfrac{ \mathrm{coth}(k\Delta h_{i-1})}{k}\Big] \dot{\pi}_i + \Big[ - \tfrac{ \mathrm{csch}(k\Delta h_{i-1})}{k}\Big] \dot{\pi}_{i-1}   \\
 &\qquad = \Big[ 2\ii U_i \mathrm{csch}(k\Delta h_i)\Big] \pi_{i+1} - \Big[ 2\ii U_i \coth(k\Delta h_i) + 2\ii U_{i-1} \mathrm{coth}(k\Delta h_{i-1})\Big] \pi_i  \\ 
 &\qquad - \Big[ kU_i^2 \mathrm{csch} (k\Delta h_i)\Big] \varphi_{i+1}  - \Big[ k^4 - kU_{i-1}^2 \mathrm{coth}(k\Delta h_{i-1}) - kU_{i}^2 \coth (k\Delta h_i) \Big] \varphi_i  \\
&\qquad + \Big[ 2\ii U_{i-1} \mathrm{csch} (k\Delta h_{i-1}) \Big] \pi_{i-1} - \Big[ kU_{i-1}^2 \mathrm{csch} (k\Delta h_{i-1})\Big]\varphi_{i-1}\end{align*}
for $2\leq i \leq n-1$, where the dot denotes differentiation with respect to $t$. These equations are supplemented with the extra equations $n$ equations $\dot{\varphi}=\pi$. The problem becomes one of linear algebra: it is straightforward to see that the governing equations will take the form:
\begin{equation} \mathcal{M}_1(k)\partial_t \begin{bmatrix} \varphi \\ \pi \end{bmatrix} = \mathcal{M}_2(k) \begin{bmatrix} \varphi \\ \pi \end{bmatrix} \label{evoeqn}\end{equation}
where $\mathcal{M}_i(k)\in \mathrm{Mat}_{2n}(\mathbf{C})$, $i=1,2$, whose entries are meromorphic in $k$. As such, if the matrices $\mathcal{M}_1$ and $\mathcal{M}_2$ are determined explicitly, one can immediately write down the solution to the initial value problem:
\begin{equation} \begin{bmatrix} \varphi \\ \pi \end{bmatrix} = \exp(t\mathcal{M})(k) \begin{bmatrix} \varphi_0 \\ \pi_0 \end{bmatrix}, \label{evoeqn2} \end{equation}
where $\mathcal{M}=\mathcal{M}_1^{-1} \mathcal{M}_2$ and the vector on the RHS contains the initial data. We see that the Cauchy problem for the linearised equations is completely determined by the matrix $\mathcal{M}(k)$, which corresponds to the infinitesimal generator of the $C_0$-semigroup for the evolution equation in the Fourier space. In what follows we prove several results regarding the matrices $\mathcal{M}_1,\mathcal{M}_2$.
\begin{lemma}\label{matriceslem}
The matrices $\mathcal{M}_1(k)$ and $\mathcal{M}_2(k)$ take the form
\[ \mathcal{M}_1(k) = \left(\begin{smallmatrix} \mathbf{I}_n & 0 \\ 0 & \mathcal{A}(k)\end{smallmatrix}\right)\qquad \textrm{and}\qquad \mathcal{M}_2(k) = \left(\begin{smallmatrix} 0 & \mathbf{I}_n \\ k\mathcal{B}(k)-k^4\mathbf{I}_n & -2\ii \mathcal{C}(k)\end{smallmatrix}\right),\]
where $\mathcal{A}(k),\mathcal{B}(k),\mathcal{C}(k) \in \mathrm{Mat}_n(\mathbf{R})$ are symmetric, tridiagonal, meromorphic in $k$ and real for $k\in \R $. The matrix $\mathcal{A}(k)$ is given by:
\[ \left[\begin{smallmatrix}\tfrac{k+\coth (k\Delta h_1) +\coth(k\Delta h_0)}{k} & - \tfrac{ \mathrm{csch}(k \Delta h_1)}{k} & 0  & \textstyle{\cdots} \\
    -  \tfrac{ \mathrm{csch}(k \Delta h_1)}{k} & \tfrac{k+\coth (k\Delta h_2)+\coth(k\Delta h_1)}{k} & - \tfrac{ \mathrm{csch}(k \Delta h_2)}{k}   &\textstyle{\cdots} \\
0 & - \tfrac{ \mathrm{csch}(k \Delta h_2)}{k} & \tfrac{k+\coth (k\Delta h_3)+\mathrm{coth}(k\Delta h_2)}{k} & \textstyle{\cdots}  \\
\vdots & \vdots & \vdots & \ddots \end{smallmatrix} \right] \]
The matrix $\mathcal{B}(k)$ is given by:
\[ \left[\begin{smallmatrix}  U_0^2 \coth(k\Delta h_0)+U_1^2 \coth (k\Delta h_1) & - U_1^2 \mathrm{csch}(k\Delta h_1)  & 0  & \textstyle{\cdots} \\
    - U_1^2 \mathrm{csch}(k\Delta h_1) &  U_1^2 \coth(k\Delta h_1)+U_2^2 \coth (k\Delta h_2) & - U_2^2 \mathrm{csch}(k\Delta h_2)   &\textstyle{\cdots} \\
0 & - U_2^2 \mathrm{csch}(k\Delta h_2) & U_2^2 \coth(k\Delta h_2)+U_3^2 \coth (k\Delta h_3) & \textstyle{\cdots}  \\
\vdots & \vdots & \vdots & \ddots \end{smallmatrix} \right] \]
The matrix $\mathcal{C}(k)$ is given by:
\[ \left[\begin{smallmatrix} U_0 \coth(k\Delta h_0)+ U_1 \coth (k\Delta h_1) & - U_1 \mathrm{csch}(k\Delta h_1)  & 0  & \textstyle{\cdots} \\
    - U_1 \mathrm{csch}(k\Delta h_1) &  U_1 \coth(k\Delta h_1)+ U_2 \coth (k\Delta h_2) & - U_2 \mathrm{csch}(k\Delta h_2)   &\textstyle{\cdots} \\
0 & - U_2 \mathrm{csch} (k\Delta h_2) & U_2 \coth(k\Delta h_2)+ U_3 \coth (k\Delta h_3) & \textstyle{\cdots}  \\
\vdots & \vdots & \vdots & \ddots \end{smallmatrix} \right] \]
\end{lemma}
Whilst the general form of $\mathcal{M}(k)$ is complicated, many of the entries become exponentially small beyond $|kh|>O(1)$, where $h=\max_i \Delta h_i$. This makes the asymptotic analysis of $\mathcal{M}(k)$ extremely simple. In dropping the exponentially small terms, the $\Delta h_i$ are removed from the problem and we deduce that the short wavelength dynamics are independent of the distribution of the elastic plates within $\Omega$. This is expected on physical grounds: the separation will effect modes with waveslengths of comparable size.

\begin{lemma}\label{inverse}
The determinant of $\mathcal{M}_1(k)$ is non-zero for each $k\in \mathbf{R}$.
\end{lemma}
\begin{proof}
It is enough to prove that $\mathcal{A}(k)$ is invertible, and we do this by proving $\mathcal{A}\succ 0$, i.e. $\mathcal{A}$ is positive definite. Denoting the diagonal entries $\{a_i\}_{i=1}^n$, and the sub/sup-diagonals by $\{-b_i\}_{i=1}^{n-1}$, we have for $X\in\mathbf{R}^n$:
\begin{align*}X\cdot ( \mathcal{A} X) &= \sum_{i=1}^n a_i X_i^2 - 2 \sum_{i=1}^{n-1} b_i X_i X_{i+1} \\
  &\geq (a_1 - b_1)X_1^2 + (a_n - b_{n-1})X_n^2 + \sum_{i=2}^{n-1} (a_i -b_i - b_{i-1})X_i^2,
\end{align*}
which follows from AM-GM inquality. It is straight forward to prove that each of the coefficients of the $X_i^2$ are strictly positive, i.e.
\[ \tfrac{1}{k} \left(k + \coth (k\Delta h_i) - \mathrm{csch}(k\Delta h_i)+\coth (k\Delta h_{i+1}) - \mathrm{csch}(k\Delta h_{i+1})\right) >0 \]
so we conclude $\mathcal{A}\succ 0$. From this we deduce that $\mathrm{det} (\mathcal{M}_1)>0$ for $k\in\mathbf{R}$.
\end{proof}
An immediate corollary to Lemma \ref{inverse} is that $\mathcal{M}_1(k)$ is invertible, and what is more, $\mathcal{M}_1^{-1}(k)$ is analytic on the real line. This follows from the previous lemma and the fact that the entries of $\mathcal{A}(k)$ are analytic on the real line.

\begin{lemma}\label{analyticprop}
The matrix $\mathcal{M}(k)$ is a analytic function of $k$ in a neighbourhood of the real line whose spectrum becomes purely imaginary as $|k|\rightarrow \infty$.
\end{lemma}
\begin{proof}
Our previous discussion establishes the analyticity of $\mathcal{M}_1^{-1}(k)$ on the real line. It is clear that $\mathcal{M}_2(k)$ is analytic away from the origin, and a straightforward computation of the relevant Laurent series about $k=0$ establishes analyticity of $(\mathcal{M}_1^{-1} \mathcal{M}_2)(k)$ at $k=0$, so the first claim is established. For the second, some standard asymptotic expansions for large, positive $k$ reveals:
\begin{subequations}\label{asymptoticM}\begin{align} \mathcal{A}(k) &\sim \mathbf{I}_n, \\ \mathcal{B}(k) &\sim \mathrm{diag}(U_0^2 + U_1^2, \ldots, U_{n-1}^2+U_n^2),\\ \mathcal{C}(k) &\sim \mathrm{diag}(U_0 + U_1, \ldots, U_{n-1}+U_n). \end{align}\end{subequations}
This gives us the asymptotic form of the characteristic polynomial:
\[ \mathrm{det} \Big(\mathcal{M}(k)-\lambda \mathbf{I}_{2n}\Big) \sim (-1)^n \prod_{m=0}^{n-1}\Big( \lambda^2 + 2\ii (U_m+U_{m+1}) - k(U_m^2+U_{m+1}^2) +k^4 \Big). \]
Since the characteristic polynomial and hence the eigenvalues of $\mathcal{M}(k)$ depend continuously on the elements of $\mathcal{M}(k)$, we deduce that the eigenvalues become purely imaginary as $k\rightarrow \infty$. A similar argument holds for $k\rightarrow -\infty$.
\end{proof}

We now address the issue of well-posedness: given initial data in a specified function space, what can be said in regards to the continuity of the map from the inital data to to the solution for times $t>0$.
\begin{theorem}[Local well-posedness in $L^2_{\dd x}(\R )^{\times 2n}$]\label{well-posed} Let the $n$ elastic plates be distributed throughout the channel $\Omega$. If the initial amplitudes $\eta_0$ and velocities $\partial_t\eta_0$ belong to $H^s_{\dd x}(\R )$, where $s\geq2$, the Cauchy problem on $\R \times (0,T)$ is locally well-posed in $L^2_{\dd x}(\R )$, so that:
 \[ \| (\eta,\partial_t \eta) \|_{L^2} \leq C_n(T) \|(\eta_0,\partial_t \eta_0)\|_{H^s} \]
where $C_n(T)$ is a constant depending on the time domain of the problem and the number of elastic plates. We have used the norm:
\[ \|(\eta,\partial_t \eta)\|_{L^2}^2 = \sum_{i=1}^n \Big(\|\eta_i\|^2_{L^2}+\|\partial_t\eta_i\|^2_{L^2}\Big) \]
and similarly for $\|\cdot\|_{H^s}$.
\end{theorem}
\begin{proof}
By Parseval's theorem it is enough to prove the relevant inequality in the Fourier space. We proceed to make a bound on $\exp(\mathcal{M}t)(k)$. We have already proved in Proposition \ref{analyticprop} that $\mathcal{M}(k)$ is analytic on the real line, so we only consider the large $k$ behaviour of $\exp(\mathcal{M}t)(k)$. From equation \eqref{asymptoticM} we see that $\mathcal{M}$ can be written as:
\[ \mathcal{P}(k) + \mathcal{E}(k) \]
where $\mathcal{P}(k)$ has purely imaginary spectrum and $\|\mathcal{E}(k)\|$ is arbitrarily small for $k$ sufficiently large\footnote{The norm $\|\cdot\|$ can be any induced norm on $\mathrm{Mat}_{2n}(\mathbf{C})$.}. Now recall the following estimate from semigroup theory: if $\|\exp(\mathcal{P}t)\|<ce^{\beta t}$ and $\mathcal{M}=\mathcal{P}+\mathcal{E}$, then:
\begin{equation} \| \exp(\mathcal{M}t)\| \leq ce^{(\beta + c\|\mathcal{E}\|)t}. \label{semigroupbound}\end{equation}
for some $c>0$. A straightforward computation reveals:
\[ \| \exp(\mathcal{P}t)\| \leq n k^2 \sqrt{t}. \]
Choosing $|k|$ large enough so that $\|\mathcal{E}\|<\epsilon$, the bound in \eqref{semigroupbound} gives:
\begin{align*} \| \exp(\mathcal{M}t)(k) \| &\leq c n k^2 \sqrt{t} e^{c\epsilon t}  \\
&\leq c n \langle k\rangle ^2 \sqrt{T} e^{c\epsilon T}.  
\end{align*}
for $k$ sufficiently large. Finally then:
\begin{align*} \| \eta\|_{L^2} &= \| \exp(t\mathcal{M})(k) \hat{\eta}_0 \|_{L^2} \\
 &\leq c n  \sqrt{T} e^{cT} \| \langle k\rangle^2 \hat{\eta}_0 \|_{L^2} \\
 & \leq C_n(T) \|\eta_0 \|_{H^s}.
\end{align*}
where $s\geq 2$. This is the desired result.
\end{proof}
\begin{figure}\label{n6_output_stable}
\begin{center}
\includegraphics[scale=0.45]{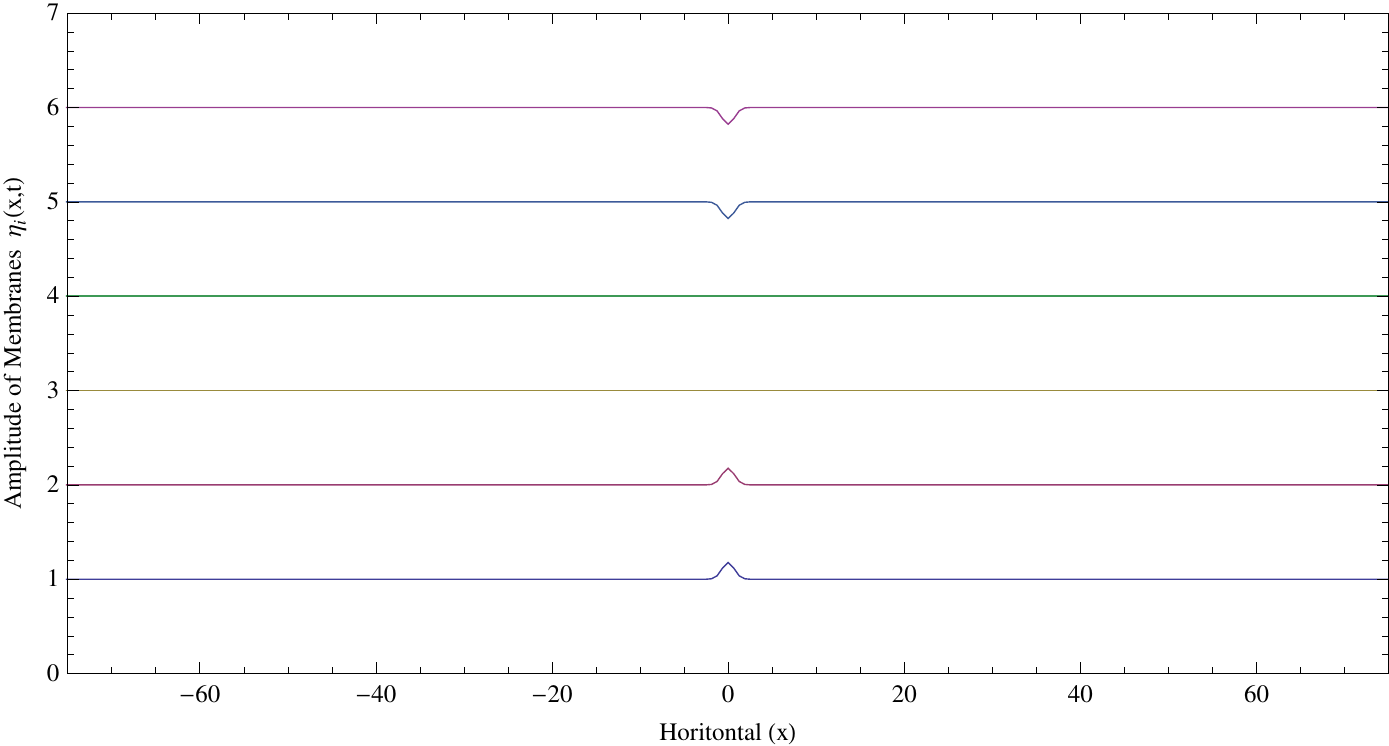}\hspace{4mm}\includegraphics[scale=0.45]{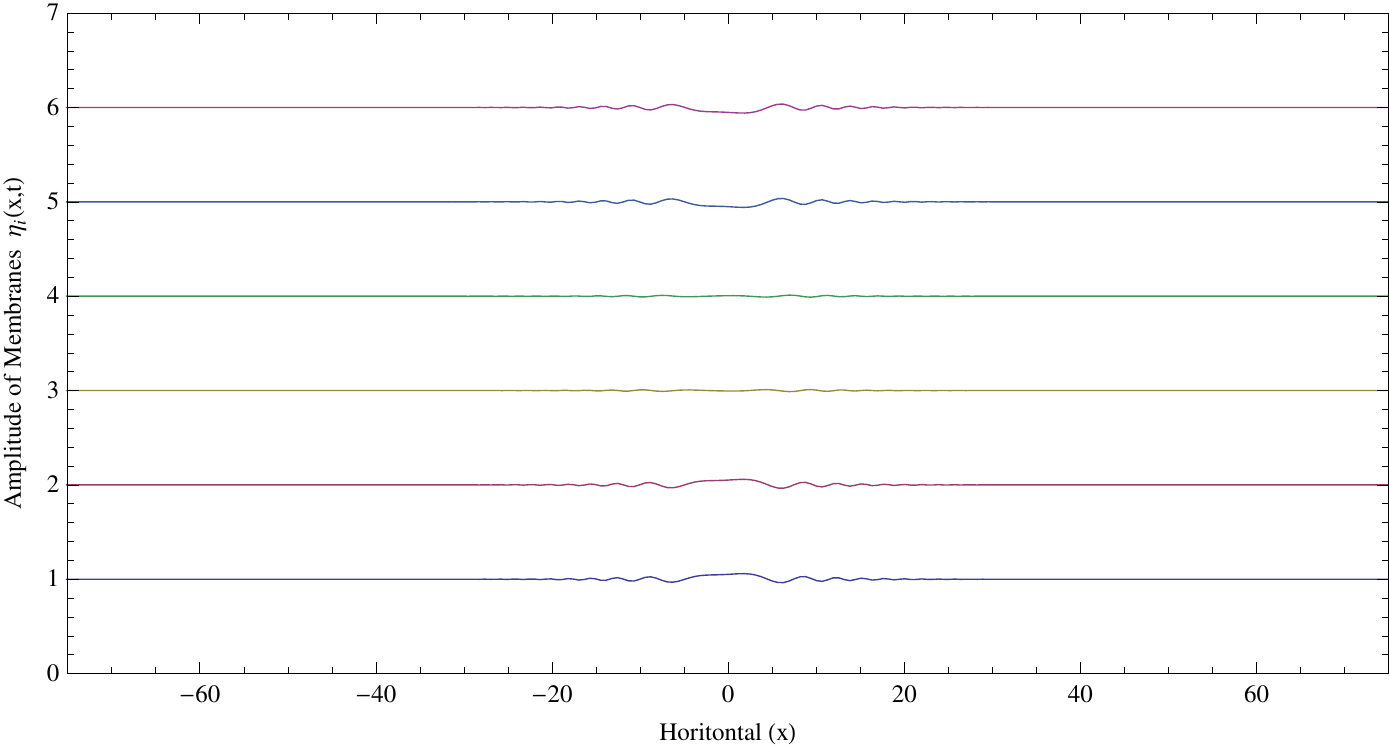}

\vspace{4mm}
\includegraphics[scale=0.45]{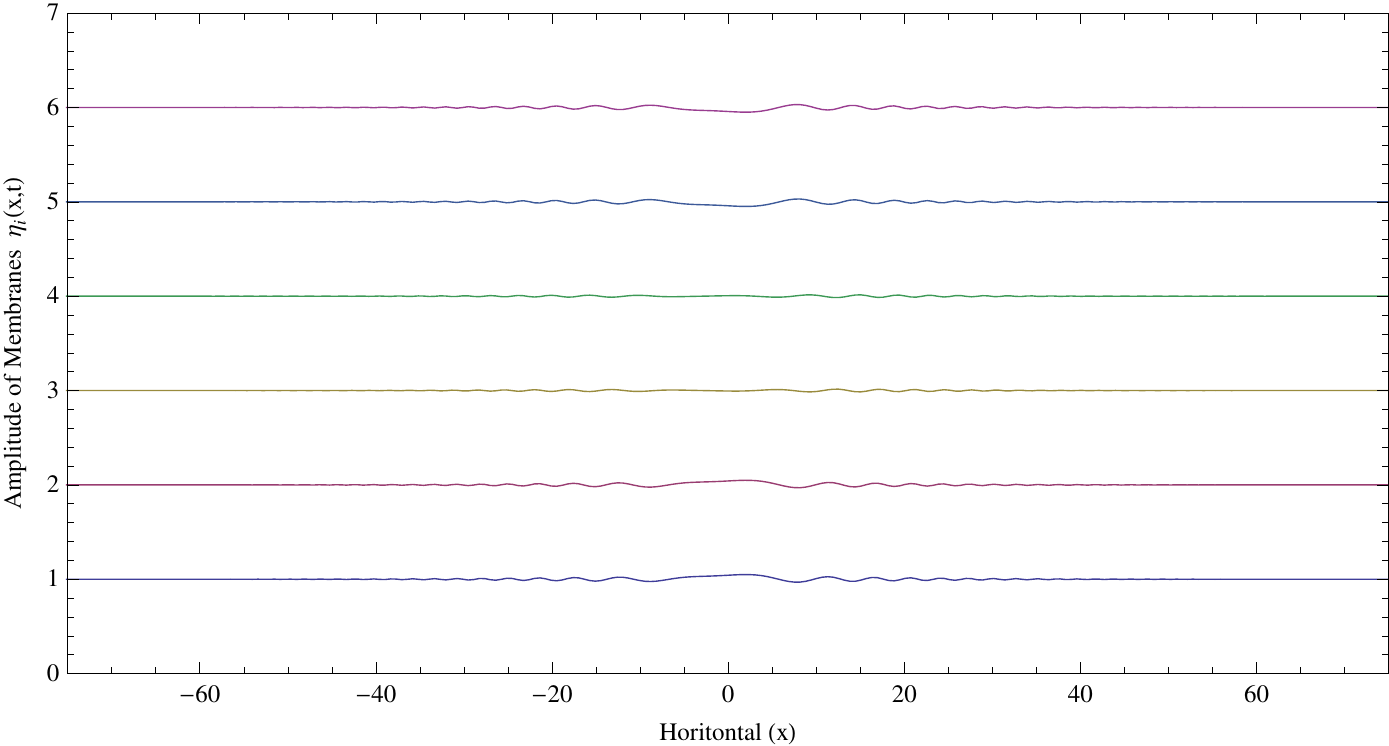}\hspace{4mm}\includegraphics[scale=0.45]{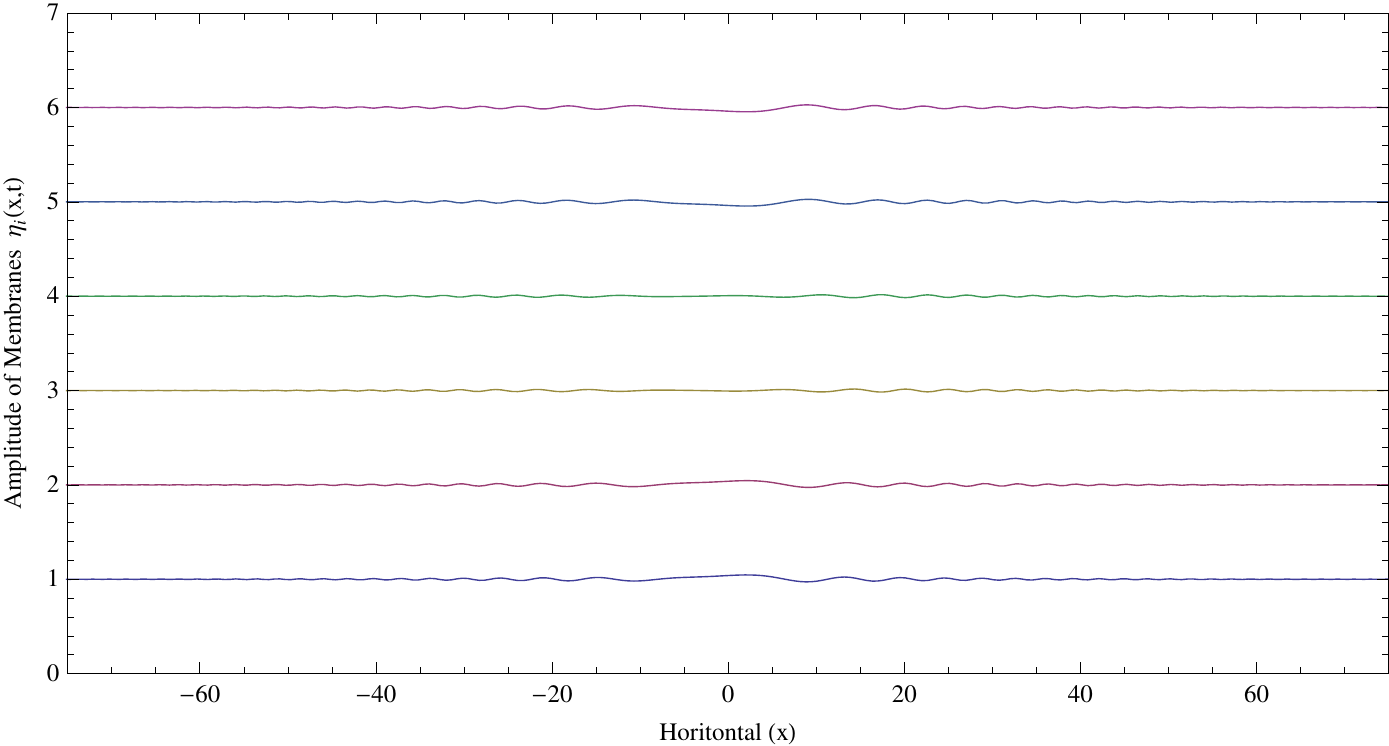}
\caption{Numerical implementation of of solution to the linear Cauchy problem with $n=6$. Four elastic plates are given Gaussian initial data, the seperation is uniform and the mean flow between each plate is $U=0.1$. The plots are evaluated at $t=0,1,2,3$ respectively. Note the amplitude of the disturbances decay, indicating short time stability of this flow configuration.}
\vspace{1cm}
\includegraphics[scale=0.45]{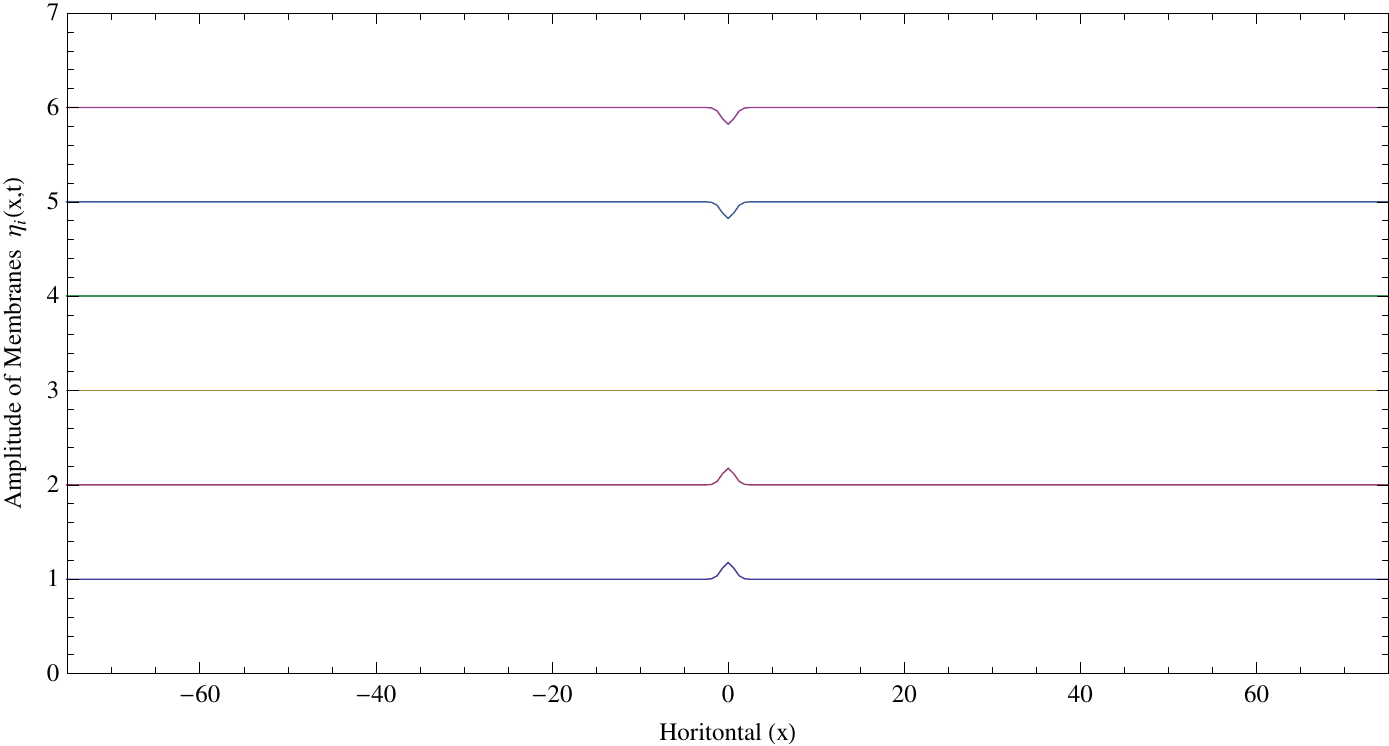}\hspace{4mm}\includegraphics[scale=0.45]{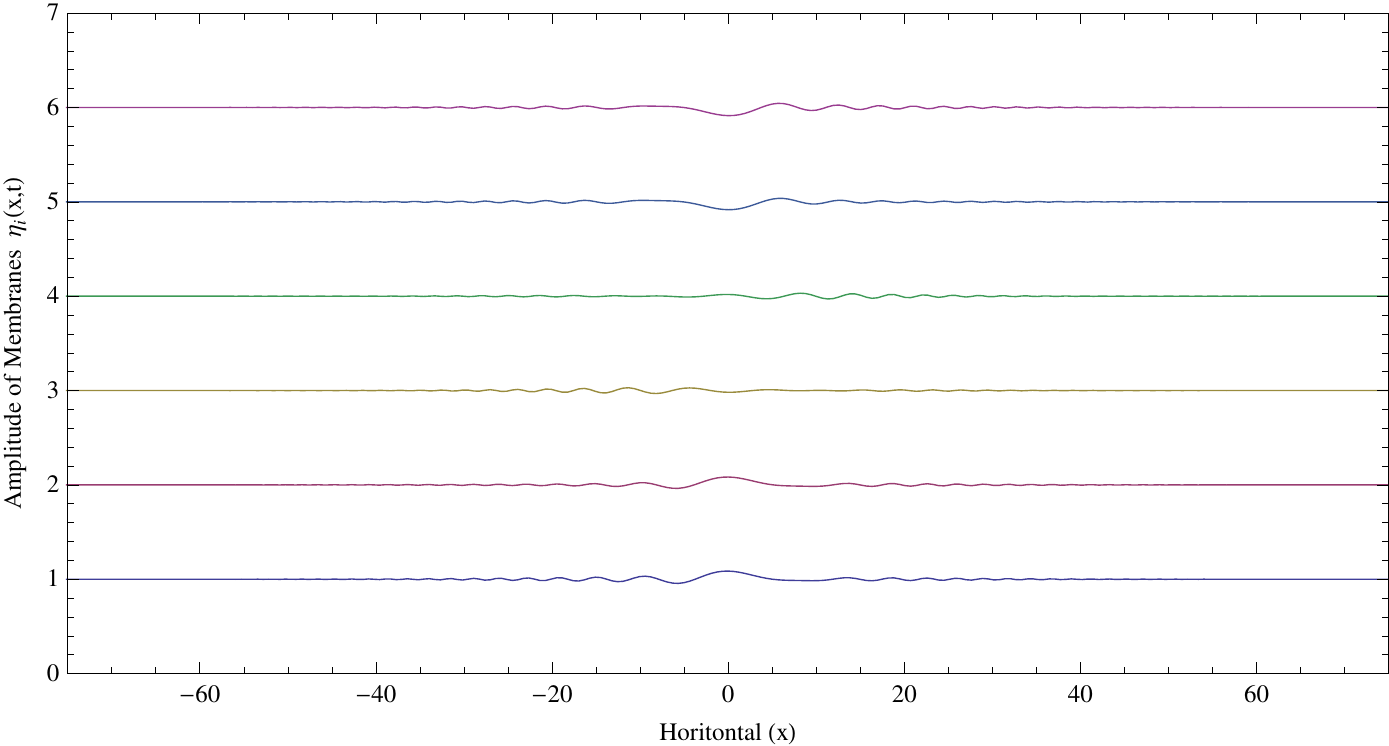}

\vspace{4mm}
\includegraphics[scale=0.45]{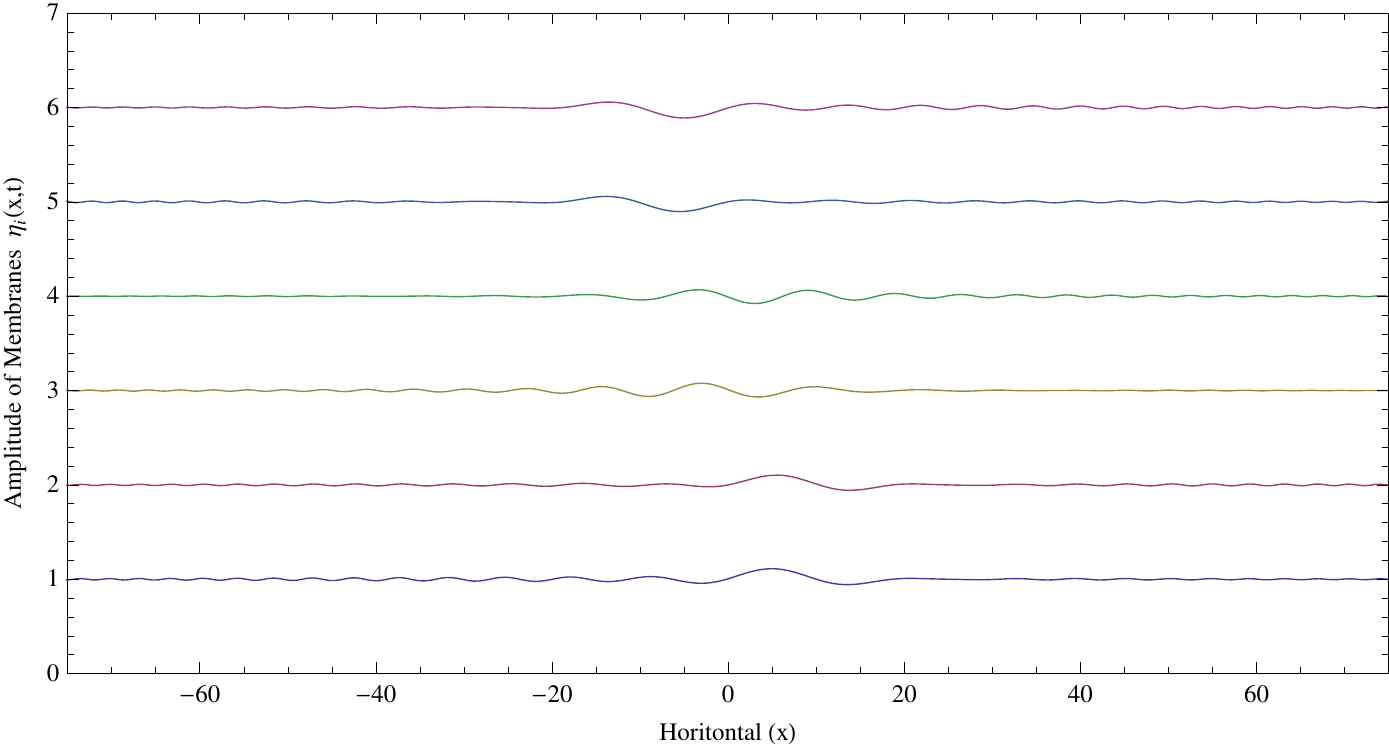}\hspace{4mm}\includegraphics[scale=0.45]{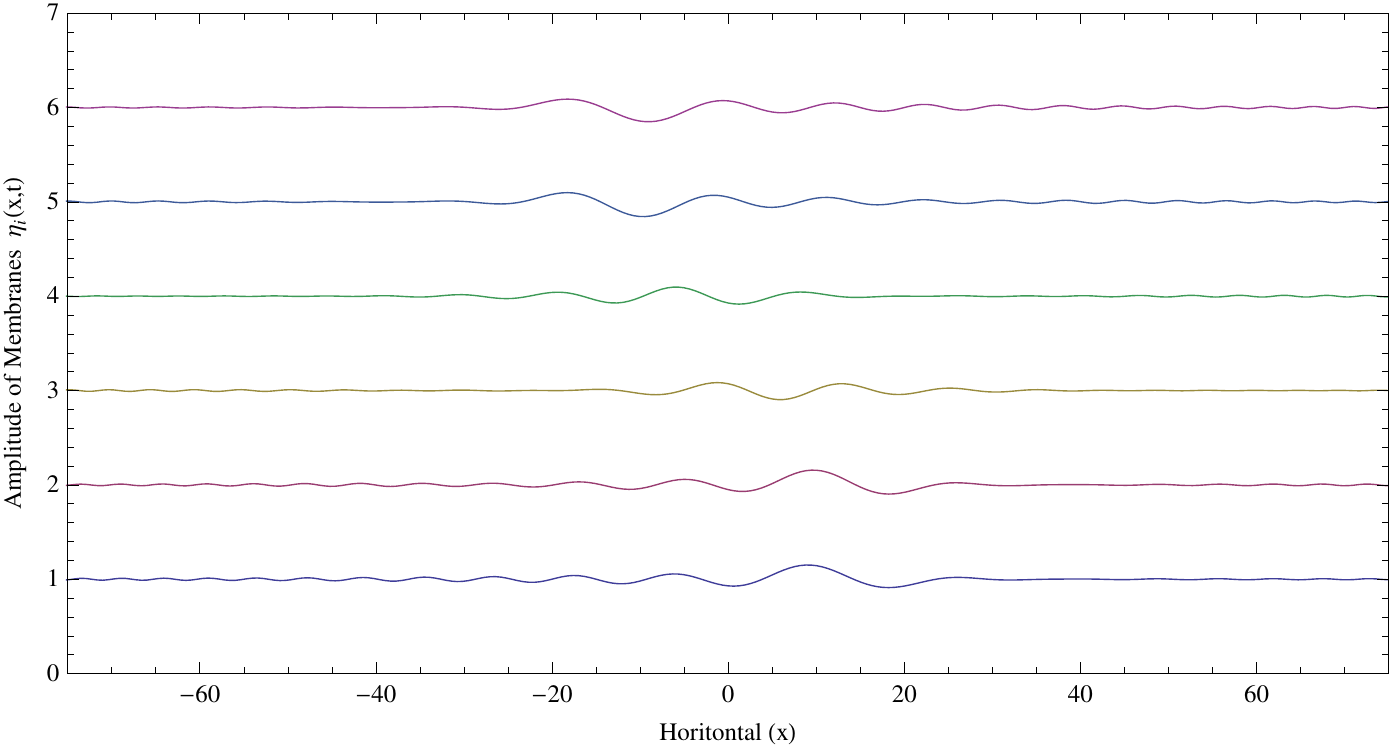}
\caption{The same numerical implementation as in Figure \ref{n6_output_stable}, but with the mean flows given by $U=(+0.4,+0.4,+0.4,0,-0.4,-0.4,-0.4)$.
In contrast to the previous case, we see the amplitutde of the disturbances increasing indicating the instability of the flow configuration.}
\end{center}
\end{figure}
This result is, of course, local. One could not hope for a global well-posedness for arbitrary mean flow velocities $U$, based purely on physical grounds. If, for instance, the magnitudes of the velocities were too great, or were in opposing directions, small instabilities would rapidly develop. See Figures 2 and 3 for an example of this phenomena.

Using standard Sobolev theory, we observe the continuous embedding:
\[ H^2_{ \dd x} (\mathbf{R}) \hookrightarrow C^1_0(\mathbf{R}), \]
where $C^1_0(\mathbf{R})$ denotes the space of continuously differentiable functions which decay at infinity. As such, we deduce that local well-posedness is achieved if the initial data is continuously differentiable with decay at infinity, possibly after redefinition on a set of measure zero.

\section{Stability}
The long time stability is determined by the Green's function of the underlying system of PDEs. Our non-local formulation has already provided us with the Green's function: from equation \eqref{evoeqn2} we see that it is given by $G \in \mathrm{Mat}_{2n}(\mathbf{C})$, where:
\begin{equation} G(x,t) =\frac{1}{2\pi} \int \exp( t\mathcal{G})(k)\, \dd k \label{expint} \end{equation}
where we have defined the matrix $\mathcal{G}\in \mathrm{Mat}_{2n}(\mathbf{C})$ by:
\begin{equation} \mathcal{G}(x,t;k) = \ii k \left(\frac{x}{t}\right) \mathbf{I}_{2n} + \mathcal{M}(k). \label{curlyG}\end{equation}
The problem is asymptotically stable if and only if $G(x,t)\rightarrow 0$ as $t\rightarrow \infty$ for each fixed $V=x/t$. It is clear that the asymptotic behaviour of $G(x,t)$ is determined by the spectrum of $\mathcal{G}$, denoted $\sigma(\mathcal{G})$. In particular, the leading order term in the asymptotic expansion will be determined by the eigenvalue with maximal real part. If $\Re \lambda > 0$ for any $\lambda\in \sigma(\mathcal{G})$, then the leading order behaviour of $G(x,t)$ has exponential growth and the problem is asymptotically unstable. In the special case where $\sigma (\mathcal{G})\in \ii \mathbf{R}$, the leading order of $G$ can be computed via the method of stationary phase, which gives $G(x,t)\sim t^{-1/2}$ as $t\rightarrow \infty$, indicating a stable configuration\footnote{Here we have assumed the relevant matrix is diagonalisable, which turns out to be the case when $\sigma(\mathcal{G})\subset \ii \mathbf{R}$.}. Alternatively one can view this as a consequence of the Riemann-Lebesgue lemma.

\begin{theorem}\label{stability}
The linear problem is asymptotically stable if and only if there is no mean flow, i.e. $U_i=0$ for $i = 0, \ldots, n$.
\end{theorem}
\begin{proof}
The problem is certainly unstable if there is at least one element of $\sigma(\mathcal{G})$ with positive real part, so it is sufficient to prove the existence of such an eigenvalue to prove the problem is unstable. We see from lemma \ref{matriceslem} that $\mathcal{M}(k)$ has the form:
\[ \begin{bmatrix} 0 & \mathbf{I}_{n} \\ k\left(\mathcal{A}^{-1}\mathcal{B}\right)(k) - k^4 \mathcal{A}^{-1}(k) & -2\ii (\mathcal{A}^{-1}\mathcal{C})(k) \end{bmatrix}. \]
Since $(\mathcal{A}^{-1}\mathcal{C})(k)\in \mathrm{Mat}_{n}(\mathbf{R})$, we see that $\Re \left( \mathrm{Tr} \mathcal{M}\right)=0$ and so from \eqref{curlyG}:
\[ \Re \left( \mathrm{Tr}\, \mathcal{G}\right)\,\, \equiv \sum_{\lambda \in \sigma(\mathcal{G})} \Re \lambda = 0. \]
It follows that either (a) all eigenvalues of $\mathcal{G}$ are purely imaginary for each $k\in\mathbf{R}$ or (b) the real parts of at least two members of $\sigma(\mathcal{G})$ have opposite sign for at least one value of $k\in\mathbf{R}$. If (b) is true, we are done, since there must be at least one member of $\sigma(\mathcal{G})$ which has positive real part. 

Now consider the case in which (a) holds, so that $\sigma(\mathcal{G}) \subset \ii \mathbf{R}$, or equivalently $\sigma(\mathcal{M})\subset \ii \mathbf{R}$. For $\lambda \in \sigma(\mathcal{M})$, set $\lambda = \ii \mu$, where $\mu \in \mathbf{R}$. A straightforward computation gives the characteristic polynomial for $\mathcal{M}$:
\begin{equation} \chi_{\mathcal{M}}(\ii\mu) = \mathrm{det} \left( \mu^2\mathcal{A}+ 2\mu \mathcal{C} + k \mathcal{B}- k^4 \mathbf{I}_n  \right). \label{detmatrix} \end{equation}
If we show that the matrix in parenthesis appearing in \eqref{detmatrix} is positive definite for some range of $k$ (for all real $\mu$), then $|\chi_{\mathcal{M}}(\ii \mu)|>0$ in that range, for all real $\mu$, and so there are no purely imaginary eigenvalues. Using a similar argument to that in the proof of lemma \ref{inverse}, one finds the matrix appearing in \eqref{detmatrix} is positive definite for all $\mu\in\mathbf{R}$ if:
\begin{multline} \big[ U_i \tanh\left(\tfrac{1}{2}k\Delta h_i\right) + U_{i-1} \tanh\left(\tfrac{1}{2}k\Delta h_{i-1}\right)\big]^2 \\
 < \big[-k^3 +  U_{i}^2 \tanh\left(\tfrac{1}{2}k\Delta h_i\right)+U_{i-1}^2 \tanh\left(\tfrac{1}{2}k\Delta h_{i-1}\right)\big] \\
\times \big[ k + \tanh\left(\tfrac{1}{2}k\Delta h_i\right) + \tanh\left(\tfrac{1}{2}k\Delta h_{i-1}\right) \big] \label{ineq}
\end{multline}
for each $i=1,\ldots , n$. This comes from examining the discriminant of the relevant quadratic in $\mu$. There is always some non-empty $k$-interval in which \eqref{ineq} is satisfied (choose $k$ sufficiently small), \emph{unless} $U_i=0$ for $i = 0, \ldots, n$, in which case \eqref{ineq} cannot be satisfied for any $k$. We deduce that $\sigma(\mathcal{M})$ always contains an element with positive real part, unless $U_i=0$ for $i = 0, \ldots, n$ in which case:
\[ \sigma(\mathcal{M})= \left\{ \pm \ii \frac{k^2}{\sqrt{\lambda(k)}}: \lambda (k) \in \sigma(\mathcal{A}) \right\}, \]
which is well defined since $\mathcal{A}\succ 0$, by lemma \ref{inverse}. In addition, we note that $\mathcal{A}$ is a Jacobi matrix so has $n$ distinct eigenvalues. It follows that in the case $U_i=0$ for $i = 0, \ldots, n$, $\mathcal{G}(k)$ has $2n$ purely imaginary, distinct eigenvalues. In particular, $\mathcal{G}(k)$ is diagonalizable for each $k\in\mathbf{R}$. We conclude that the problem is asymptotically stable if and only if there is no flow, in which case $G(x,t) \sim t^{-1/2}$ as $t\rightarrow \infty$, with $V=x/t$ fixed.
\end{proof}
This result can be reconciled with the well-posedness result in theorem \ref{well-posed}. In the case that the $U_i$ are not all zero, we know that $\sigma(\mathcal{M})$ is not purely imaginary for all $k\in\mathbf{R}$, but rather $\sigma(\mathcal{M})$ is becomes purely imaginary for $|k|$ large. As such, the optimal bound on $\|\exp (\mathcal{M} t)\|$ was of the form in \eqref{semigroupbound}, which led to the \emph{local} well-posedness result. If we consider the case $U_i=0$ for $i = 0, \ldots, n$, then $\sigma(\mathcal{M}) \subset \ii \mathbf{R}$, and the optimal bound in this case is found by using:
\begin{equation} \| \exp(\mathcal{M}t) \| \leq \kappa (\mathcal{M}) e^{\alpha (\mathcal{M}) t}, \label{betterbound} \end{equation}
where $\alpha(\mathcal{M})$ is the spectral abscissa of the matrix $\mathcal{M}$, defined by:
\[ \alpha(\mathcal{M}) = \sup_{\lambda \in \sigma(\mathcal{M})} \Re \lambda \]
and $\kappa (\mathcal{M})=\|\mathcal{V}\| \| \mathcal{V}^{-1}\|$ is the condition number of $\mathcal{M}$, where $\mathcal{V}$ is the matrix of eigenvectors of $\mathcal{M}$. After some standard asymptotic estimates, one can show:
\[ \mathcal{V}(k) \sim \begin{bmatrix} \ii k^2 \mathbf{I}_n & - \ii k^2 \mathbf{I}_n \\ \mathbf{I}_n & \mathbf{I}_n \end{bmatrix}. \]
Using this, it is straightforward to show that $\kappa (\mathcal{M})$ is $\mathcal{O}(k^2)$ for large $k$. Finally, using the estimate in \eqref{betterbound} and the fact that $\alpha(\mathcal{M})=0$, we have:
\[ \| \exp(\mathcal{M} t)(k) \| \leq C_n \langle k \rangle^2 \]
where $C_n$ is a constant depending only on $n$. In summary, for that case of no mean flow, so that $U_i=0$ for $i = 0, \ldots, n$, we have following global well-posedness result.

\begin{theorem}[Global well-posedness in $L^2_{\dd x}(\mathbf{R})^{\times 2n}$ in the case of zero mean flow]
 Let the $n$ elastic plates be distributed throughout the channel $\Omega$, and suppose there is no mean flow so that $U_i=0$ for $i = 0, \ldots, n$. Then if the initial amplitudes $\eta_0$ and velocities $\partial_t\eta_0$ belong to $H^s_{\dd x}(\R )$, with $s\geq2$, the Cauchy problem is globally well-posed in $L^2_{\dd x}(\R )$. 
\end{theorem}

Let us now examine the result in Theorem \ref{stability}. We see that the condition that $\Re \sigma(\mathcal{G})=\emptyset$ is only violated for small $k$, i.e. only modes corresponding to long wavelengths are unstable. Let $K$ denote the interval for which \eqref{ineq} is satisfied, so that for $k\in K$ we know $\Re\,\sigma(\mathcal{M})\neq \emptyset$. To understand the instabilities it is enough to know $\mathcal{M}$ for $k\in K$, i.e. $k$ small. It is straightforward, although algebraically intensive, to show that the leading order terms in the real part of the eigenvalues of $\mathcal{M}$ are of the form:
\begin{equation} \frac{\mathcal{L}(U,\Delta h)}{|\Omega|} k + \mathcal{O}(k^2). \label{leadingorder} \end{equation}
where $\mathcal{L}$ is a bilinear function of $U=(U_0, U_1,\ldots)$ and $\Delta h=(\Delta h_0, \Delta h_1, \ldots)$ and $|\Omega|\equiv h_{n+1}-h_0$ is the channel width. For large seperations $\Delta h$, and mean flows with small velocities, this quantity is exceedingly small. As such, the growth rate of the amplitudes of the corresponding modes will be very slow. One would expect to have to wait for long times before such instabilities were realised, as indicated in Figure \ref{n6_output_stable}.

%\begin{figure} \label{expansion}
%\[ \left(
%\begin{smallmatrix}
% 0 & 0 & 0 & 1 & 0 & 0 \\
% 0 & 0 & 0 & 0 & 1 & 0 \\
% 0 & 0 & 0 & 0 & 0 & 1 \\
% 0 & 0 & 0 & -\frac{2 \ii k \left(h_1 U_0+h_2 U_0+h_3 U_0+h_0 U_1\right)}{h_0+h_1+h_2+h_3} & \frac{2 \ii k \left(h_0 U_1-h_0 U_2\right)}{h_0+h_1+h_2+h_3} &
%   \frac{2 \ii h_0 k \left(U_2-U_3\right)}{h_0+h_1+h_2+h_3} \\
% 0 & 0 & 0 & \frac{2 \ii k \left(-h_2 U_0-h_3 U_0+h_2 U_1+h_3 U_1\right)}{h_0+h_1+h_2+h_3} & -\frac{2 \ii k \left(h_2 U_1+h_3 U_1+h_0 U_2+h_1
%   U_2\right)}{h_0+h_1+h_2+h_3} & \frac{2 \ii k \left(h_0 U_2+h_1 U_2-h_0 U_3-h_1 U_3\right)}{h_0+h_1+h_2+h_3} \\
% 0 & 0 & 0 & -\frac{2 \ii h_3 k \left(U_0-U_1\right)}{h_0+h_1+h_2+h_3} & -\frac{2 \ii k \left(h_3 U_1-h_3 U_2\right)}{h_0+h_1+h_2+h_3} & -\frac{2 \ii k
%   \left(h_3 U_2+h_0 U_3+h_1 U_3+h_2 U_3\right)}{h_0+h_1+h_2+h_3}
%\end{smallmatrix}
%\right)\]
%\caption{
%Leading order term for the expansion of $\mathcal{M}(k)$ as $k\in K$, in the case $n=3$.}
%\end{figure} 

We have shown that the linear problem is asymptotically stable if and only if there is no mean flow, however it is unlikely that in practice this stability will be realised. It is well known that the naive treatment of stability by means of spectra can be misleading, the more relevant quantity being the pseudospectrum:
\[ \sigma_\epsilon (\mathcal{M}) = \{ \lambda \in \mathbf{C} : \| (\mathcal{M}-\lambda \mathbf{I})^{-1}\| > \epsilon^{-1} \}. \]
The pseudospectrum contains information about how the spectrum of the associated operator changes under small perturbations, in the sense that $\sigma_\epsilon (\mathcal{M})$ contains information about the spectrum of the operator $\mathcal{M}+\mathcal{E}$, where $\|\mathcal{E}\|<\epsilon$. See figure \ref{pseudo_output} for some plots of $\sigma_\epsilon (\mathcal{M})$ for given $\epsilon$. Recall that the pseudospectrum of a normal operator $\mathcal{N}$ is:
\[ \bigcup_{\lambda \in \sigma(\mathcal{N})} B_\epsilon(\lambda), \]
i.e. one draws balls of radius $\epsilon$ about each element of the spectrum of $\mathcal{N}$. In this sense, the deviation from this form of pseudospectrum is a measure of non-normality. See \cite{trefethen2005spectra} for several different, but equivalent definitions of pseudospectra and non-normality. It is clear from Figure \label{pseudo_output} that for $k=1$ the operator is approximately normal, and as $|k|$ decreases the non-normality of $\mathcal{M}$ increases. Indeed, the width $\sigma_\epsilon(\mathcal{M})$ is approximately $8\epsilon$ in the final output. This indicates that whilst theoretically the instabilities of the long-wavelength modes grow very slowly, in practice the instabilities will grow at a much faster rate since $\Re \, \sigma_\epsilon(\mathcal{M})$ approaches $\mathcal{O}(1)$ for small $k$.
\begin{figure}\label{pseudo_output}
\begin{center}
\includegraphics[scale=0.47]{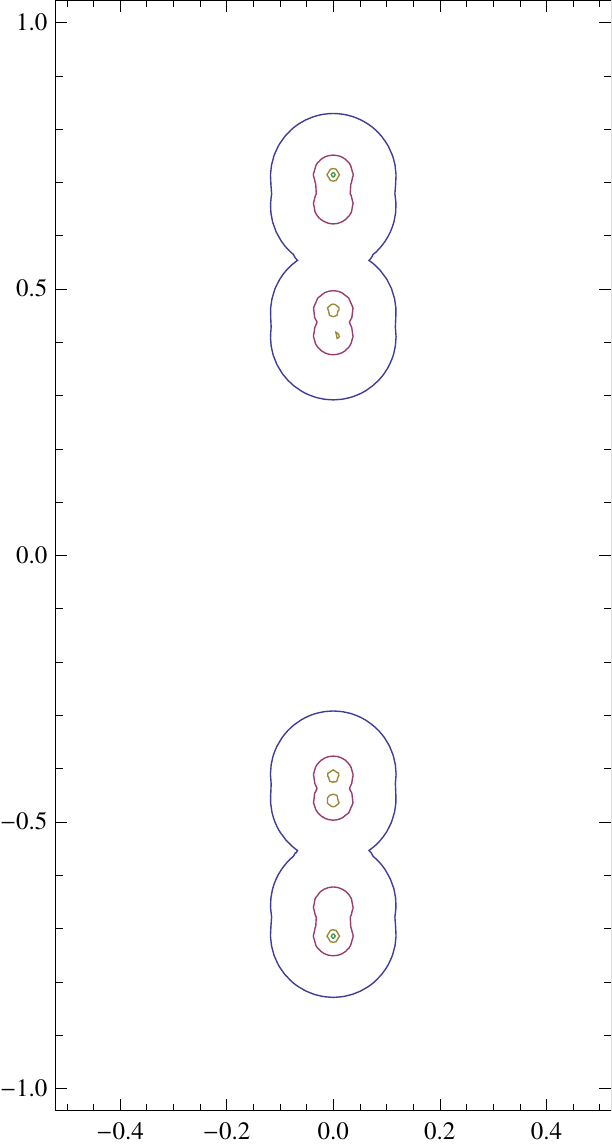}\hspace{4mm}\includegraphics[scale=0.47]{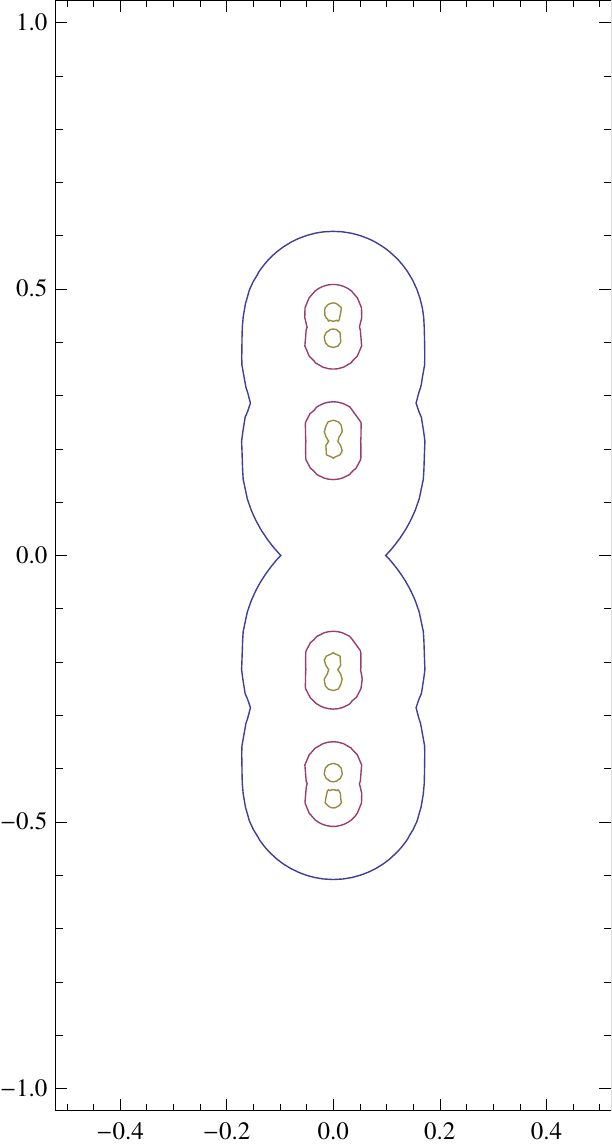}\hspace{4mm}\includegraphics[scale=0.47]{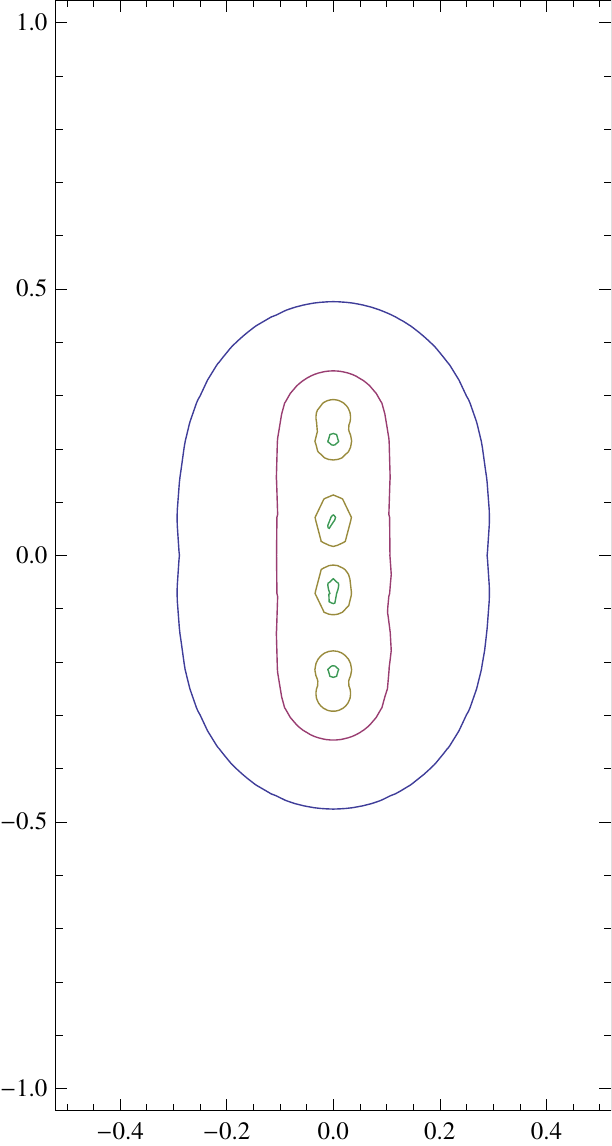}\hspace{4mm}\includegraphics[scale=0.47]{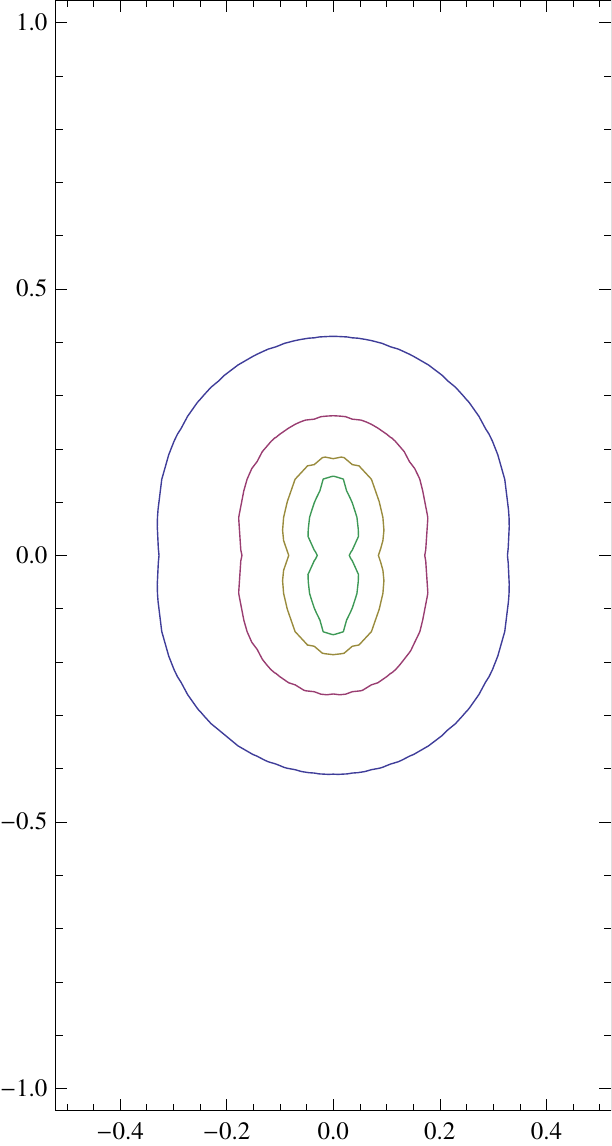}
\caption{Plots of the pseudospectrum of $\mathcal{M}(k)$ with $n=4$, for $k$ decreasing from $1$ to $0$. Plate seperations are all $10$, and $U=(+0.15,+0.3,0,-0.3,-0.15)$. The plots are given by $\|(\mathcal{M}-\lambda \mathbf{I})^{-1}\|=\epsilon^{-1}$, where $\epsilon = 10^{-1}, 10^{-1.5}, 10^{-2},10^{-2.5}$. The width of $\sigma_\epsilon (\mathcal{M})$ increases as $k$ decreases, indicating the non-normality of the matrix $\mathcal{M}$ for small $k$.}
\end{center}
\end{figure}

At this stage we recall a result of Chrighton \& Oswell \cite{crighton1991flm}, who proved that in the case of a single elastic plate driven by one mean flow (above) the configuration was convectively stable for $U<U_c =0.074$. However, the configuration in their work differs considerably from that considered here -- most importantly the flow is not bound above (i.e. their channel is infinite in width). In our problem, the fluid in $\Omega_i$ transfers vertical momentum betwen $\eta_i$ and $\eta_{i+1}$, and in particular the rigid boundaries of the channel reflect back waves carrying vertical momentum. In the case of non-zero mean flow, as the fluid interacts with the plates it is able increase the vertical component of the momentum -- this process continues and instabilities are liable to grow in time. However, in \cite{crighton1991flm} the infinite vertical extend of the channel means that disturbances can ``escape'' from the channel, since the fluid can carry vertical momentum away from the plate and out to infinity, and the aforementioned incease in vertical momentum does not occur. One can heuristically reconcile their work with ours by taking $|\Omega|\rightarrow \infty$ in \eqref{leadingorder}.

\section{Variations of the Problem}
In this section we briefly discuss some variations and extensions of our problem. For the sake of brevity, some of the details are omitted, but all the would-be calculations are similar to those seen in the prequal.

\subsection{An Infinite Channel}
In some certain circumstances it is more appropriate to treat the surfaces $\mathcal{B}_\pm$ at infinity. In this case, the analysis changes on the top and bottom domains. Recall that in the non-local formulation on the top domain, the initial step \eqref{divint_top} involved the integral:
\begin{multline*} \int_{\Gamma_n} e^{-\ii k x + \kappa y}(\kappa  \partial_x \phi_i-\ii k  \partial_y\phi_i ,\kappa  \partial_y\phi_i + \ii k \partial_x\phi_i)\cdot N(\Gamma_n)\, \dd x \\
- \int_{\mathcal{B}^+}  e^{-\ii k x + \kappa y}(\kappa  \partial_x \phi_i-\ii k  \partial_y\phi_i ,\kappa  \partial_y\phi_i + \ii k \partial_x\phi_i)\cdot N(\mathcal{B}^+)\, \dd x =0.
\end{multline*}
In the case of the infinite channel, we are required to take the limit $y\rightarrow \infty$ in the second integral and as such we need to impose $\kappa = k<0$. In this case, the integral equation in the top domain $\Omega_n$ becomes:
\begin{equation} \int e^{-\ii k x + k\eta_n}\Big(  \partial_{U_n}\eta_n  + \ii  \partial_x\xi_n^- \Big)\, \dd x =0, \label{k<0} \end{equation}
with the restriction $k<0$. Alternatively, changing $k\mapsto -k$, \eqref{k<0} becomes:
\begin{equation} \int e^{+\ii k x - k\eta_n}\Big(  \partial_{U_n}\eta_n  +\ii  \partial_x\xi_n^- \Big)\, \dd x =0, \qquad k>0. \label{infinitetop} \end{equation}
The corresponding equation for the bottom domain $\Omega_0$ is:
\begin{equation} \int e^{-\ii k x +k\eta_1}\Big( \partial_{U_0}\eta_1  +\ii  \partial_x\xi_0^+ \Big)\, \dd x =0, \qquad k>0.  \label{infinitebottom}\end{equation}
The analysis on the internal domains will be exactly the same as in the case of the finite channel. The restrictions that $k>0$ in equations \eqref{infinitetop} and \eqref{infinitebottom} does not cause any restrictions, at least in the linear case, because a \emph{real} function can be recovered from the knowledge of its Fourier transform on $k>0$. This follows from the observation:
\begin{align*} \int e^{\ii kx} \hat{f}(k)\, \dd k &= \int_0^{\infty} \left[ e^{\ii kx} \hat{f}(k) + e^{-\ii kx}\hat{f}(-k)\right] \dd k \\
 &= \int_0^\infty \left[ e^{\ii kx} \hat{f}(k) + e^{-\ii kx} \overline{\hat{f}(k)}\,\right ] \dd k,
\end{align*}
which is a consequence of the reality of $f$ and the definition of the Fourier transform. Since the functions in our problem are manifestly real, we see that the restriction that $k>0$ on the top and bottom domains is permissable. A similar restriction was made in \cite{ashton2008nms} in which the authors prove well-posedness results for the initial value problem for a single fluid loaded elastic plate on both the full and half lines.

\subsection{Non-zero Vorticity}
In the presence of non-zero vorticity, one cannot invoke a scalar velocity potential, so the methods outlined thus far are inappropriate. In this case, one refers to the stream function $\psi$ defined so that $\psi_y$ represents the horizontal velocity, and $-\psi_x$ represents the vertical velocity of the fluid. For a given vorticity $\omega$, the stream function satisfies:
\begin{equation} -\Delta \psi = \omega. \label{harmvort}\end{equation}
In \cite{ashton2008nfr} the authors derive a non-local formulation for rotational water waves with constant vorticity $\omega=\gamma$, by introducing the pseudo-potential $\varphi$ defined as the harmonic conjugate to the function:
\[ \psi^h = \psi + \tfrac{1}{2} \gamma y^2. \]
which is harmonic by virtue of equation \eqref{harmvort}. The problem can be stated in terms of the harmonic function $\varphi$, and a non-local formulation follows in a similar manner as that seen in \S 1. 

\subsection{Higher Dimensions}
An important generalisation of the problem is to extend it to higher dimensions, three spatial dimensions being the most physically relevant case. We analyse the problem in $m+1$ spatial dimensions, with $m>1$, using the coordinates:
\[ (x_1, \ldots, x_m, y) \in \R ^{m+1} \]
where $y$ corresponds to the vertical direction. The elastic plates are now described by the $m$ dimensional hypersurfaces:
\[ \Gamma_i = \{ (x,y) \in \R ^{m+1}: y=\eta_i (x,t) \}, \] 
where $x=(x_1, \ldots, x_m)$ is the ``horizontal'' coordinate. The mean flow in each $\Omega_i$ is horizontal and can be represented by $(U_i,0)\in \R ^{m+1}$. The scalar differential operator $\partial_U$ needs a suitable modifcation in the higher dimensional setting -- it becomes:
\[ \partial_U = \partial_t + U\cdot \nabla\! _x \]
where $\nabla\! _x =(\partial_{x_1}, \ldots, \partial_{x_m},0)$ is the horizonal gradient operator. In addition, the beam equation for each elastic plate gets modified so that:
\[ \partial_t^2 \eta_i + \Delta_x^2 \eta_i = p_{i-1}-p_i \qquad \textrm{on $\Gamma_i$}, \]
where $\Delta_x$ is the laplacian in the horizontal coordinates. The relevant kinematic boundary conditions are modified in an obvious fashion and by introducing analogues of the $\{\xi_i^\pm\}$, one can express the derivatives of $\phi$ in terms of $\eta$ and $\xi$. Indeed, the $\{\nabla\! _x \phi_i\}$ satisfy the following non-singular equation:
\begin{equation} \nabla\! _x \xi_i^- - (\partial_{U_i} \eta_i) \nabla\! _x \eta_i = (\mathbf{I} + \nabla\! _x \eta_i \otimes \nabla\! _x \eta_i) \nabla\! _x \phi_i \qquad \textrm{on $\Gamma_i$}, \label{highdim} \end{equation}
which comes from eliminating $\partial_y \phi_i$ from the higher dimensional analogues of equations \eqref{chainrules(i)a} and \eqref{kinematic_pair}. From this we can express $\partial_t \phi_i$ and $\partial_y\phi_i$ in terms of $\xi^-_i$ and $\eta_i$ on the surface $\Gamma_i$. The higher dimensional analogue of the identity in lemma \ref{global_lem} is:
\[ \nabla_x \cdot \Big((\partial_y u )\nabla_x v+ (\partial_y v) \nabla_x u\Big) + \partial_y \Big( (\partial_y u)( \partial_y v) - \nabla_x u \cdot \nabla_x v\Big) =0 \]
which holds for all harmonic functions $u,v$ on $\mathbf{R}^{m+1}$. Replacing $v$ by $\exp(\ii k\cdot x \pm \|k\|y)$, with $k\in\R^m$, the governing equations \eqref{alleqns}, \eqref{alleqnsup}, \eqref{alleqnsdown}, in the case of $m+1$ spatial dimensions, can be reformulated in terms of a non-local problem.
\begin{proposition}\label{summary_prop_(n+1)}
The solution to the boundary value problem described by \eqref{alleqns}, \eqref{alleqnsup} and \eqref{alleqnsdown} in the case of $m+1$ spatial dimensions is completely determined by the $3n+2$ functions $\{\eta_i\}_{i=1}^n$, $\{\xi_i^\pm \}_{i=0}^n$ which satisfy:
\begin{itemize}
 \item The $2n-2$ non-local \textbf{internal equations}:
\begin{multline}
 \int e^{-\ii k\cdot x} \Big( \partial_{U_i}\eta_i \sinh (\|k\|\eta_i) + \ii \hat{k}\cdot \nabla\! _x \xi_i^- \cosh (\|k\|\eta_i) \\ - \partial_{U_i} \eta_{i+1} \sinh(\|k\|\eta_{i+1}) - \ii \hat{k}\cdot \nabla\! _x\xi_i^+\cosh(\|k\|\eta_{i+1}) \Big)\dd x =0, \label{summaryint1_(n+1)}\end{multline}
\begin{multline}
 \int e^{-\ii k\cdot x} \Big( \partial_{U_i}\eta_i \cosh (\|k\|\eta_i) + \ii \hat{k}\cdot \nabla\! _x \xi_i^- \sinh (\|k\|\eta_i) \\ - \partial_{U_i} \eta_{i+1} \cosh(\|k\|\eta_{i+1}) - \ii \hat{k}\cdot \nabla\! _x\xi_i^+\sinh(\|k\|\eta_{i+1}) \Big)\dd x =0, \label{summaryint2_(n+1)}
\end{multline}
where $\hat{k}$ is the unit vector in the $k$-direction and $\dd x \equiv \dd x_1 \wedge \cdots \wedge \dd x_m$ is the volume form on $\R ^m$. These are valid for $k\in \R ^m$ and $1\leq i \leq n-1$.
 \item The $4$ non-local \textbf{external equations}:
\begin{subequations}\label{summaryintegrals_(n+1)}
\begin{multline}
 \int e^{-\ii k\cdot x} \Big( \partial_{U_n}\eta_n \sinh (\|k\|\eta_n) + \ii \hat{k}\cdot \nabla\! _x \xi_n^- \cosh (\|k\|\eta_n)  \\  - \ii \hat{k}\cdot \nabla\! _x\xi_n^+\cosh(\|k\|(h_0+h_+)) \Big)\dd x =0, \label{summaryint1-n_(n+1)}\end{multline}
\begin{multline}
 \int e^{-\ii k\cdot x} \Big( \partial_{U_n}\eta_n \cosh (\|k\|\eta_n) + \ii \hat{k}\cdot \nabla\! _x \xi_n^- \sinh (\|k\|\eta_n)  \\  - \ii \hat{k}\cdot \nabla\! _x\xi_n^+\sinh(\|k\|(h_0+h_+)) \Big)\dd x =0, \label{summaryint2-n_(n+1)}\end{multline}
\begin{multline}
 \int e^{-\ii k\cdot x} \Big( \partial_{U_0}\eta_1 \sinh (\|k\|\eta_1) + \ii \hat{k}\cdot \nabla\! _x \xi_0^+ \cosh (\|k\|\eta_1)  \\  - \ii \hat{k}\cdot \nabla\! _x\xi_0^-\cosh(\|k\|(h_0-h_-)) \Big)\dd x =0, \label{summaryint1-0_(n+1)}\end{multline}
\begin{multline}
 \int e^{-\ii k\cdot x} \Big( \partial_{U_0}\eta_1 \cosh (\|k\|\eta_1) + \ii \hat{k}\cdot \nabla\! _x \xi_0^+ \sinh (\|k\|\eta_1)  \\  + \ii \hat{k}\cdot \nabla\! _x\xi_0^-\sinh(\|k\|(h_0-h_-)) \Big)\dd x =0, \label{summaryint2-0_(n+1)}\end{multline}
which are valid for $k\in\R ^m$. 
\end{subequations}
 \item The $n$ \textbf{Bernoulli-type equations}:
\begin{equation} \partial_t^2\eta_i + \Delta_x^2 \eta_i + \partial_{U_{i-1}}\phi_{i-1} + \tfrac{1}{2}\|\nabla\! \phi_{i-1}\|^2   =\partial_{U_i}\phi_i+\tfrac{1}{2} \|\nabla\!  \phi_i\|^2 \qquad \textrm{on $\Gamma_i$}\label{summarybernoulli_(n+1)} \end{equation}
for $1\leq i \leq n$, where $\nabla \phi_i|_{\Gamma_i}$ and $\partial_t \phi_i |_{\Gamma_i}$ etc. are given in terms of $\xi_i^\pm$ via \eqref{highdim} and analogues thereof.
\end{itemize}
\end{proposition}
Using this result it is straightforward to prove analogous results to those seen in \S 3 and \S 4. We quickly arrive at the following lemma, which is analogous to lemma \ref{matriceslem} in the case of two spatial dimensions.
\begin{lemma}\label{matricesdim}
In $m+1$ spatial dimensions, the linearised equations take the form:
\[ \mathcal{N}_1(k) \partial_t \begin{bmatrix} \varphi \\ \pi \end{bmatrix} = \mathcal{N}_2(k) \begin{bmatrix} \varphi \\ \pi \end{bmatrix}. \]
The matrices $\mathcal{N}_1(k)$ and $\mathcal{N}_2(k)$ take the form
\[ \mathcal{N}_1(k) = \left(\begin{smallmatrix} \mathbf{I}_n & 0 \\ 0 & \mathcal{P}(k)\end{smallmatrix}\right)\qquad \textrm{and}\qquad \mathcal{N}_2(k) = \left(\begin{smallmatrix} 0 & \mathbf{I}_n \\ \mathcal{Q}(k)-\|k\|^4\mathbf{I}_n & -2\ii \mathcal{R}(k)\end{smallmatrix}\right),\]
where $\mathcal{P}(k),\mathcal{Q}(k),\mathcal{R}(k) \in \mathrm{Mat}_n(\mathbf{R})$ are symmetric, tridiagonal and real for $k\in \R^m $. The matrix $\mathcal{P}(k)$ is given by:
\[ \left[\begin{smallmatrix}\tfrac{\|k\|+\coth (\|k\|\Delta h_1) +\coth(\|k\|\Delta h_0)}{\|k\|} & - \tfrac{ \mathrm{csch}(\|k\| \Delta h_1)}{\|k\|}   & \textstyle{\cdots}\\
    -  \tfrac{ \mathrm{csch}(\|k\| \Delta h_1)}{\|k\|} & \tfrac{\|k\|+\coth (\|k\|\Delta h_2)+\coth(\|k\|\Delta h_1)}{\|k\|} & \textstyle{\cdots} \\
\vdots & \vdots & \ddots \end{smallmatrix} \right]. \]
The matrix $\mathcal{Q}(k)$ is given by:
\[ \left[\begin{smallmatrix}  \tfrac{(k\cdot U_0)^2 \coth(\|k\|\Delta h_0)+(k\cdot U_1)^2 \coth (\|k\|\Delta h_1)}{\|k\|} & - \tfrac{(k\cdot U_1)^2 \mathrm{csch}(\|k\|\Delta h_1)}{\|k\|}    & \textstyle{\cdots}\\
    - \tfrac{(k\cdot U_1)^2 \mathrm{csch}(\|k\|\Delta h_1)}{\|k\|} &  \tfrac{(k\cdot U_1)^2 \coth(\|k\|\Delta h_1)+(k\cdot U_2)^2 \coth (\|k\|\Delta h_2)}{\|k\|} & \textstyle{\cdots} \\
 \vdots & \vdots & \ddots \end{smallmatrix} \right]. \]
The matrix $\mathcal{R}(k)$ is given by:
\[ \left[\begin{smallmatrix}  \tfrac{(k\cdot U_0) \coth(\|k\|\Delta h_0)+(k\cdot U_1) \coth (\|k\|\Delta h_1)}{\|k\|} & - \tfrac{(k\cdot U_1) \mathrm{csch}(\|k\|\Delta h_1)}{\|k\|}    & \textstyle{\cdots} \\
    - \tfrac{(k\cdot U_1) \mathrm{csch}(\|k\|\Delta h_1)}{\|k\|} &  \tfrac{(k\cdot U_1) \coth(\|k\|\Delta h_1)+(k\cdot U_2) \coth (\|k\|\Delta h_2)}{\|k\|} & \textstyle{\cdots} \\
 \vdots & \vdots & \ddots \end{smallmatrix} \right]. \]
\end{lemma}
It is clear form of the matrices $\mathcal{P}(k)$, $\mathcal{Q}(k)$ and $\mathcal{R}(k)$ that the results of \S 4 \& 5 will follow in a similar fashion, and as such we simply state the following.
\begin{theorem}[Well-posedness in higher dimensions]
Let the $n$ elastic plates be distributed horiztonally throughout the channel $\Omega \subset \mathbf{R}^{m+1}$. If the initial amplitudes $\eta_0$ and velocities $\partial_t\eta_0$ belong to $H^s_{\dd x}(\R^m)$, with $s\geq2$, then the Cauchy problem is locally well-posed in $L^2_{\dd x}(\R^m)$. If, in addition, there is zero mean flow, so that that $U_i=0$ for $i = 0, \ldots, n$, the problem is corresponding Cauchy problem is globally well-posed.
\end{theorem}
\begin{corollary}
The linear problem in $m+1$ spatial dimensions is asymptotically stable if and only if there is no mean flow, i.e. $U_i=0$ for $i = 0, \ldots, n$.
\end{corollary}
These results are very much similar to that seen in the two dimensional case. The one difference is the conditions needed on the initial data to ensure well-posedness. For example, in the physically interesting case of three dimensions, we only need the initial data to be continuous and drop the need for differentiablity. This is a consequence of the well-known embedding:
\[ H^2_{\dd x} (\mathbf{R}^2) \hookrightarrow C^0_0(\mathbf{R}^2). \]
In higher dimensions, the conditions are even weaker: it is enough for the initial data to belong to $L^{2m/(m-4)}_{\dd x}(\mathbf{R}^m)$ for $m>4$. The case $m=4$ is special, and in this instance the initial data need be in $L^p_{\dd x}(\mathbf{R}^4)$ for each $p\in [1,\infty)$ to ensure well-posedness.

The fact that the linear problem is unstable for non-zero fluid speeds is no great surprise. Again, the fluid carries vertical momentum between the plates, which will be reflected back from the rigid boundary of the channel $\Omega$. If the fluid has non-zero velocity, and the initial data is non-zero, then the fluid can increase the vertical momentum of the plates, and this process is able to continue, allowing instabilities to develop.

\subsection{Periodic Motion}
Problems of physical interest will always involve finite spatial domains, so in certain circumstances treatment of the problem on the infinite line is inappropriate. For this reason, it is important to be able the analyse the case in which the plates are of finite length -- one way of achieving this is by considering periodic motion.

It is well known that for a bounded flows (i.e. flows on compact domains), there is in general only a \emph{countable} infinity of normal modes. As such, one expects the relevant global relations, i.e. the integral equations in Proposition \ref{summary_prop} etc, to be valid for $k$ in some countable subset of $\R $, and this is indeed the case. 

For simplicity assume all functions have period $2\pi$. In the non-local formulation, after constructing a divergence free expression we integrated it over the entire domain $\Omega_i$ (see equation \eqref{divint} and those thereafter). In the periodic case, the appropriate adjustment is to integrate over one cell in $\Omega_i$, of width $2\pi$. In the case of the infinite line, the contribution from the vertical components of $\partial\Omega_i$ vanished due to our assumptions on the decay of the fields. In the periodic case, we only have the fact that all the functions agree on either side of the cell, so to eliminate the corresponding integrals, we must choose $k$ so that: 
\[ \exp(-\ii kx) = \exp(-\ii kx -2\pi \ii k), \]
i.e. we choose $k\in \mathbf{Z}$. The modified version of Proposition \ref{summary_prop} would now involve integral equations valid for $k\in\mathbf{Z}$. Of course, this is sufficient\footnote{With moderate a priori conditions on the relevant function spaces.} to rebuild the functions $\eta_i$ and $\xi^\pm_i$. Indeed, we know from standard Fourier analysis on $L^2_{\mathrm{per}}$ that the functions $\eta_i$, $\xi_i^\pm$ can be reconstructed from their relevant Fourier coefficients, so the global relations, valid for $k\in \mathbf{Z}$, are simply equations for these coefficients. In the case of the infinite line, the global relations can be treated as equations for the Fourier transforms $\hat{\eta}_i$ and $\hat{\xi}_i^\pm$. In fact, the global relations can be viewed as pseudodifferential operators acting on the Fourier transforms, but we shall not elaborate on this point in this paper.

\section{Conclusions}
In this paper we have given a detailed study of the linear problem, involving the motion of a collection of elastic plates driven by different mean flows. In particular, we study the necessary and sufficient conditions needed for asymptotic stability, and also the well-posedness of the associated Cauchy problem. In the latter case, we find sufficient conditions on the initial amplitudes and speeds of the plates so that the solution of the problem changes continuously with the initial data.

Our approach involves a non-local formalism, in which the underlying system of nonlinear equations are converted to a coupled system of 1-parameter integral equatoins. This approach is largely motivated by a new, unified approach to boundary value problems introduced by A.S Fokas in the late 1990's \cite{fokas2008uab}.

In the two dimensional problem, we show that it is enough for the initial data to be continuously differentiable with decay at infinity, for the problem to remain well-posed. We then extended the results so that they are applicable to the problem in higher dimensions. Using some standard Sobolev embedding results, we find that in the case of three spatial dimensions, the need for differentiability can be dropped, and it is sufficient that the initial data be continuous (up to a set of measure zero).

Finally, we study the asymptotic stability of the problem: for which configurations does the Green's function decay to zero after long times, in a frame moving at arbitrary velocity? We find that if there is any mean flow, the problem is asymptotically unstable, although the unstable modes grow very slowly if the channel is large, and the fluid velocities are small. We offer support of this result on physical grounds, based on momentum transfer between the fluid and the plates. In the case in which there is no mean flow, we show that the Green's function decays like $\mathcal{O}(t^{-1/2})$ for large $t$, which follows from analysis of the spectrum of the infinitesimal generator of the semigroup for the Cauchy problem in the Fourier space. We also study the pseudospectrum of this operator, and discover that for long wavelength modes, the operator becomes highly non-normal. This indicates that whilst theoretically the configuration should remain stable for zero mean flow, and instabilities should only develop after long time periods in the case of slow flow, in practice it is likely that these instabilities will develop at a much faster rate.

\newpage
\nocite{whitham1974lan}
\bibliography{ants_bib}
\bibliographystyle{plain}

\end{document}

%% file: pic.pdf_t
\begin{picture}(0,0)%
\includegraphics{pic.pdf}%
\end{picture}%
\setlength{\unitlength}{2072sp}%
\begingroup\makeatletter\ifx\SetFigFontNFSS\undefined%
\gdef\SetFigFontNFSS#1#2#3#4#5{%
  \reset@font\fontsize{#1}{#2pt}%
  \fontfamily{#3}\fontseries{#4}\fontshape{#5}%
  \selectfont}%
\fi\endgroup%
\begin{picture}(7272,6813)(1339,-6997)
\put(4771,-3526){\makebox(0,0)[lb]{\smash{{\SetFigFontNFSS{8}{9.6}{\rmdefault}{\mddefault}{\updefault}{\color[rgb]{0,0,0}$\Omega_i$}%
}}}}
\put(4771,-6406){\makebox(0,0)[lb]{\smash{{\SetFigFontNFSS{8}{9.6}{\rmdefault}{\mddefault}{\updefault}{\color[rgb]{0,0,0}$\Omega_0$}%
}}}}
\put(4771,-736){\makebox(0,0)[lb]{\smash{{\SetFigFontNFSS{8}{9.6}{\rmdefault}{\mddefault}{\updefault}{\color[rgb]{0,0,0}$\Omega_n$}%
}}}}
\put(8596,-466){\makebox(0,0)[lb]{\smash{{\SetFigFontNFSS{8}{9.6}{\rmdefault}{\mddefault}{\updefault}{\color[rgb]{0,0,0}$\mathcal{B}^+$}%
}}}}
\put(8596,-1141){\makebox(0,0)[lb]{\smash{{\SetFigFontNFSS{8}{9.6}{\rmdefault}{\mddefault}{\updefault}{\color[rgb]{0,0,0}$\Gamma_n$}%
}}}}
\put(8596,-2491){\makebox(0,0)[lb]{\smash{{\SetFigFontNFSS{8}{9.6}{\rmdefault}{\mddefault}{\updefault}{\color[rgb]{0,0,0}$\Gamma_{i+1}$}%
}}}}
\put(8596,-4741){\makebox(0,0)[lb]{\smash{{\SetFigFontNFSS{8}{9.6}{\rmdefault}{\mddefault}{\updefault}{\color[rgb]{0,0,0}$\Gamma_i$}%
}}}}
\put(8596,-6091){\makebox(0,0)[lb]{\smash{{\SetFigFontNFSS{8}{9.6}{\rmdefault}{\mddefault}{\updefault}{\color[rgb]{0,0,0}$\Gamma_1$}%
}}}}
\put(8596,-6901){\makebox(0,0)[lb]{\smash{{\SetFigFontNFSS{8}{9.6}{\rmdefault}{\mddefault}{\updefault}{\color[rgb]{0,0,0}$\mathcal{B}^-$}%
}}}}
\put(6481,-3031){\makebox(0,0)[lb]{\smash{{\SetFigFontNFSS{7}{8.4}{\rmdefault}{\mddefault}{\updefault}{\color[rgb]{0,0,0}$-N(\Gamma_{i+1})$}%
}}}}
\put(6976,-4021){\makebox(0,0)[lb]{\smash{{\SetFigFontNFSS{7}{8.4}{\rmdefault}{\mddefault}{\updefault}{\color[rgb]{0,0,0}$N(\Gamma_i)$}%
}}}}
\put(2566,-3526){\makebox(0,0)[lb]{\smash{{\SetFigFontNFSS{8}{9.6}{\rmdefault}{\mddefault}{\updefault}{\color[rgb]{0,0,0}$U_i$}%
}}}}
\put(2566,-781){\makebox(0,0)[lb]{\smash{{\SetFigFontNFSS{8}{9.6}{\rmdefault}{\mddefault}{\updefault}{\color[rgb]{0,0,0}$U_n$}%
}}}}
\put(2566,-6451){\makebox(0,0)[lb]{\smash{{\SetFigFontNFSS{8}{9.6}{\rmdefault}{\mddefault}{\updefault}{\color[rgb]{0,0,0}$U_0$}%
}}}}
\end{picture}%

%% file: nonlinear_flp.bbl
\begin{thebibliography}{10}

\bibitem{ashton2007fkf}
A.C.L. Ashton.
\newblock The fundamental k-form and global relations.
\newblock {\em SIGMA 4}, 033:15 pages; arXiv:0711.4707, 2008.

\bibitem{ashton2008nfr}
A.C.L. Ashton and A.S. Fokas.
\newblock {A Nonlocal Formulation of Rotational Water Waves}.
\newblock {\em Arxiv preprint arXiv:0812.2316}, 2008.

\bibitem{ashton2008nms}
A.C.L. Ashton and A.S. Fokas.
\newblock A novel method of solution for the fluid loaded plate.
\newblock {\em Arxiv preprint arXiv:0812.1660}, 2008.

\bibitem{crighton1991flm}
D.G. Crighton and J.E. Oswell.
\newblock Fluid loading with mean flow. i. response of an elastic plate to
  localized excitation.
\newblock {\em Philosophical Transactions: Physical Sciences and Engineering},
  335(1639):557--592, 1991.

\bibitem{fokas1997utm}
A.S. Fokas.
\newblock A unified transform method for solving linear and certain nonlinear
  pdes.
\newblock {\em Proceedings: Mathematical, Physical and Engineering Sciences},
  453(1962):1411--1443, 1997.

\bibitem{fokas2008uab}
A.S. Fokas.
\newblock {\em {A Unified Approach to Boundary Value Problems}}.
\newblock CBMS 78, SIAM, 2008.

\bibitem{jia2007cmb}
L.B. Jia, F.~Li, X.Z. Yin, and X.Y. Yin.
\newblock {Coupling modes between two flapping filaments}.
\newblock {\em Journal of Fluid Mechanics}, 581:199--220, 2007.

\bibitem{peake2001nonlinear}
N.~Peake.
\newblock {Nonlinear stability of a fluid-loaded elastic plate with mean flow}.
\newblock {\em Journal of Fluid Mechanics}, 434:101--118, 2001.

\bibitem{trefethen2005spectra}
L.N. Trefethen and M.~Embree.
\newblock {\em {Spectra and pseudospectra: the behavior of nonnormal matrices
  and operators}}.
\newblock Princeton Univ Pr, 2005.

\bibitem{whitham1974lan}
G.B. Whitham.
\newblock {\em Linear and nonlinear waves}.
\newblock New York, 1974.

\end{thebibliography}
